\documentclass[11pt]{amsart}

\usepackage{amsmath,amssymb,amsthm,amscd}
\usepackage{graphicx,subfigure}

\makeatletter
\newcommand{\newreptheorem}[2]{\newtheorem*{rep@#1}{\rep@title}\newenvironment{rep#1}[1]{\def\rep@title{#2 \ref{##1}}\begin{rep@#1}}{\end{rep@#1}}}
\makeatother

\newtheorem{thm}{Theorem}[section]
\newtheorem*{thm*}{Theorem}
\newreptheorem{thm}{Theorem}
\newtheorem{lem}[thm]{Lemma}
\newtheorem*{lem*}{Lemma}
\newtheorem{prop}[thm]{Proposition}
\newtheorem*{prop*}{Proposition}

\theoremstyle{definition}

\newcommand{\Ad}[1]{\text{Ad}_{#1}}
\newcommand{\HP}{\text{HP}}
\newcommand{\PSL}{\operatorname{PSL}}
\newcommand{\PGL}{\operatorname{PGL}}
\newcommand{\SO}{\operatorname{SO}}
\newcommand{\Sol}{\text{Sol}}
\newcommand{\AdS}{\text{AdS}}

\newcommand{\Hom}{\operatorname{Hom}}

\title{Hyperbolic structures from Sol on pseudo-Anosov mapping tori}
\author{Kenji Kozai}
\address{Department of Mathematics,
University of California, Berkeley,
970 Evans Hall \#3840,
Berkeley, CA 94720-3840}
\email{kozai@math.berkeley.edu}
\thanks{}

\begin{document}

\begin{abstract}
The invariant measured foliations of a pseudo-Anosov homeomorphism
induce a natural (singular) Sol structure on mapping tori of surfaces
with pseudo-Anosov monodromy. We show that when the pseudo-Anosov
$\phi:S\rightarrow S$ has orientable foliations and does not have 1 as an
eigenvalue of the induced
cohomology action on the closed surface, then the Sol structure can be
deformed to nearby cone hyperbolic structures, in the sense of projective
structures. The cone angles can be chosen to be decreasing from multiples
of $2\pi$.
\end{abstract}

\maketitle

\section{Introduction}

Let $S=S_{g,n}$ be a surface of genus $g$ with $n$ punctures such that
$2g+n>2$.
Given a homeomorphism $\phi:S \rightarrow S$, we can define the
mapping torus $M_\phi = S \times [0,1] / (x,1) \sim (\phi(x),0)$.
The hyperbolization theorem by Thurston \cite{thurston98} states that
$M_\phi$ is hyperbolic if and only if $\phi$ is pseudo-Anosov.
A pseudo-Anosov
homeomorphism $\phi:S \rightarrow S$ has two transverse (possibly singular)
foliations
$\mathcal{F}^s$ and $\mathcal{F}^u$ with transverse measures $\mu_s$ and
$\mu_u$, respectively, and a constant $\lambda>1$ such that $\phi$ preserves
$\mathcal{F}^s$ and $\mathcal{F}^u$ and scales the measures by
$\lambda^{-1}$ and $\lambda$. When $S$ is not closed, the map $\phi$
induces a pseudo-Anosov map on the closed surface $\bar{S}$ of genus $g$,
where the $n$ punctures have been filled in. We will also call this map
$\phi:\bar{S} \rightarrow \bar{S}$.

The measured foliations $(\mathcal{F}^s,\mu_s)$ and
$(\mathcal{F}^u,\mu_u)$ endow $S$ with a singular Euclidean metric. The
corresponding suspension flow $\phi_t$ on $M_\phi$, expanding the leaves of
$\mathcal{F}^u$ by a factor of $e^t$ and contracting the leaves of
$\mathcal{F}^s$ by $e^{-t}$, has period $\log \lambda$, so that
$\phi_{\log\lambda}=\phi$. One model for Sol
geometry is to take $\mathbb{R}^3$ with the metric $ds^2 = e^{2z}dx^2+
e^{-2z}dy^2+dz^2$, so the suspension flow can be viewed as an isometry
of Sol translating the surface $S$ in the $z$ direction. The
identification $(x,y,z+\log\lambda) \sim (\phi(x,y),z)$ then defines a
singular Sol structure on $M_\phi$, with singular locus $\Sigma$ given by the
orbits of the singular points and punctures of $\mathcal{F}^s$ and
$\mathcal{F}^u$.

In the case where $S$ is a punctured torus, Hodgson \cite{hodgson86}
studied how to deform representations of $\pi_1(M_\phi)$ near a
representation corresponding to a projection of the Sol structure. Sol
space contains embedded hyperbolic planes, and the representations
studied in \cite{hodgson86} correspond to projecting the 3-manifold onto
a hyperbolic plane inside Sol, resulting in a reducible representation that
gives $M_\phi$ the structure of a transversely
hyperbolic foliation (recall that a representation
$\rho:\pi_1(M_\phi) \rightarrow \PSL(2,\mathbb{C})$ is \textit{irreducible}
if the only subspaces of $\mathbb{C}^2$ that are invariant under $\rho$
are trivial). Further results about deforming reducible representations to
irreducible representations can be found in \cite{frohman91}, 
\cite{heusener98}, and \cite{abdelghani02}.
Heusener, Porti, and Su{\'a}rez \cite{heusener01-2} have also shown that
hyperbolic structures can be regenerated from Sol, constructing a path of
nearby hyperbolic structures that collapse onto a circle, and rescaling the
metric as it collapses to obtain the Sol metric on $M_\phi$.

In the case where $S$ is not the punctured torus, such a
regeneration theorem is not known. In this paper, we utilize half-pipe (HP)
geometry, studied by Danciger \cite{danciger13}, to regenerate hyperbolic structures
in a more general setting. In particular, we will prove the following result.

\begin{repthm}{thm:conestructures}
	Let $\phi: S \rightarrow S$ be a pseudo-Anosov homeomorphism whose stable
	and unstable foliations, $\mathcal{F}^s$ and $\mathcal{F}^u$, are orientable
	and $\phi^*:H^1(\bar{S}) \rightarrow H^1(\bar{S})$ does not have 1 as an
	eigenvalue. Then, there exists a family
	of singular hyperbolic structures on $M_\phi$, smooth on the complement of
	$\Sigma$ and with cone singularities along $\Sigma$, that degenerate to a
	transversely hyperbolic foliation. The
	degeneration can
	be rescaled so that the path of rescaled structures limit to the singular
	Sol structure on $M_\phi$, as projective structures. Moreover, the cone angles
	can be chosen to be decreasing.
\end{repthm}

The proof of Theorem \ref{thm:conestructures} uses HP structures as an
intermediate. We find a family of HP structures that collapse, such that
rescaling the collapse in an appropriate manner yields Sol. The HP structures
involved are built from a representation $\rho_0: \pi_1(M_\phi \setminus \Sigma)
\rightarrow \PSL(2,\mathbb{C})$ arising from projecting the 3-dimensional Sol
space to one of its embedded hyperbolic planes, along with a first order
deformation of the representation. The following 
is an application of the Ehresmann--Thurston principle:

\begin{thm*}[\cite{danciger13}, Proposition 3.6]
	Let $M_0$ be a compact $n$-manifold with boundary and let $M$ be
	a thickening of $M_0$ so that $M\setminus M_0$ is a collar neighborhood of
	$\partial M_0$. Suppose $M$ has an $\HP$ structure defined by the developing
	map $D_\HP$, and holonomy representation $\sigma_\HP$. Let $X$ be either
	$\mathbb{H}^n$ or $\text{AdS}^n$ and let $\rho_t:\pi_1(M_0) \rightarrow
	\text{Isom}(X)$
	be a family of representations compatible to first order at time $t=0$ with
	$\sigma_\HP$. Then we can construct a family of $X$ structures on $M_0$ with
	holonomy $\rho_t$ for short time.
\end{thm*}

As noted in \cite{danciger13}, given an $\HP$ structure, the regeneration of
a hyperbolic structure only requires that it exists on the level of representations.
In Theorem \ref{thm:conestructures}, the
conditions that the invariant foliations $\mathcal{F}^s$ and $\mathcal{F}^u$
are orientable and that $\phi^*$ does not have $1$ as an eigenvalue
guarantee smoothness of the representation variety at $\rho_0$, so we can
find a nearby family of representations $\rho_t$. We also
do a simple computation to generalize Danciger's notion of infinitesimal cone
angle to multiple components. This allows us to adapt the HP machinery to
show that there are singular hyperbolic structures near the HP structures,
which are themselves collapsing to the Sol structure. We will then show that
the singular locus can be controlled so that the family of $\mathbb{H}^3$
structures are cone manifolds.

\subsection{Outline}

In Section \ref{sec:background}, we present an overview of geometric
structures and infinitesimal deformations. Section \ref{sec:metabelian}
describes the collapsed structure as a metabelian representation and
establishes the notation used in the following section. Section
\ref{sec:smoothness} proves smoothness of the representation variety at the
metabelian representation, which is used in Section \ref{sec:singularstructures}
to show that we can find nearby three dimensional hyperbolic structures via
HP geometry. Section \ref{sec:singularlocus} analyzes the behavior of the
singular locus to show that the singularities can be realized as cone
singularities, providing the final step to Theorem \ref{thm:conestructures}.

\subsection{Acknowledgments}

The author would like to thank Steven Kerckhoff for advising much of this work
at Stanford University and Jeffrey Danciger for many useful conversations
about $\HP$ structures. The author would also like to thank the reviewer for
helpful comments and references.

\section{Background}\label{sec:background}

Let $X$ be a manifold and $G$ be a group of analytic diffeomorphisms of $X$.
We will study
geometric structures on a manifold $M$ through the framework of
$(X,G)$-structures described by Ehresmann \cite{ehresmann36} and Thurston
\cite{thurston80}.

\subsection{$(X,G)$ structures}

An \textit{$(X,G)$ structure} on a manifold $M$ is a collection of charts
$\{\psi_\alpha : U_\alpha \rightarrow X\}$, where the $\{U_\alpha\}$ are an
open cover of $M$ and the transition maps
$\psi_\alpha \psi_\beta^{-1}$ are restrictions of elements
$g_{\alpha \beta} \in G$.

In the context of this paper, we will take $X$ to be (a subset of)
$\mathbb{R}P^3$ and $G$ to be (a subgroup of) $\PGL(4,\mathbb{R})$, with
$\mathbb{H}^3$ and Sol being described as projective structures. An $(X,G)$
structure on $M$ defines a developing map $D:\tilde{M}\rightarrow X$ that is
equivariant under the holonomy representation $\rho:\pi_1(M) \rightarrow X$.

A smooth family of $(X,G)$-structures on a manifold $M$ can be described
by a family of developing maps $D_t: \tilde{M}\rightarrow X$ and corresponding
holonomy representations $\rho_t : \pi_1(M) \rightarrow G$. 
Two families of $(X,G)$-structures $D_t$ and $F_t$ such that $D_0 = F_0$
are equivalent if there exists a smooth family $g_t$ of elements in $G$ and
a smooth family of diffeomorphisms $\phi_t$ defined on all but a 
neighborhood of 
$\partial M$ such that $D_t = g_t \circ F_t \circ \tilde{\phi_t}$ where
$\tilde{\phi_t}$ is the lift of $\phi_t$,
$g_0 = 1$, and $\tilde{\phi_0}$ is the identity. 
Such a deformation $D_t$ is trivial if $D_0$ is equivalent to the family of
structures $F_t = D_0$. In this case, the holonomy representations also differ
by conjugation by a smooth family $g_t$, i.e. $\rho_t = g_t \rho_0 g_t^{-1}$.

We will study deformations of geometric structures through their
representations.
Let $R(\pi_1(M),G)=\Hom(\pi_1(M),G)$ be the variety of
representations of $\pi_1(M)$ into $G$, $\mathcal{X}(\pi_1(M),G) = R(\pi_1(M),G)//G$
be the character variety, where the quotient is the GIT quotient as $G$ acts
by conjugation, and let $\mathcal{D}(M,(X,G))$ be the space of
$(X,G)$-structures on $M$ up to the equivalence defined. The
Ehresmann--Thurston principle states that locally, deformations of
geometric structures can be studied by their holonomy representations
(see \cite{goldman88} for a proof of the theorem).

\begin{thm*}[Thurston]\label{thm:hol}
	The map $\text{hol} : \mathcal{D}(M,(X,G)) \rightarrow
	\mathcal{X}(\pi_1(M),G)$ taking an $(X,G)$ structure to its holonomy
	representation is a local homeomorphism on $\text{hol}^{-1}(
	\Hom(\pi_1(M),G)^{st}/G)$, where $\Hom(\pi_1(M),G)^{st}$ is the
	subset of $\Hom(\pi_1(M),G)$ consisting of stable orbits.
\end{thm*}

Given a smooth family of
representations $\rho_t : \pi_1(M) \rightarrow G$, we can study
the infinitesimal change in $\rho_t$ at $\rho_0$, as in \cite{hodgson86}.
The derivative of the homomorphism condition
$\rho_t (ab) = \rho_t(a) \rho_t(b)$ yields
\begin{equation*}
	\rho_t'(ab) = \rho_t'(a)\rho_t(b) + \rho_t(a) \rho_t'(b).
\end{equation*}
In order to normalize the derivative, we multiply on the right by
$\rho_t(ab)^{-1}$ to translate back to the identity element to obtain
\begin{equation*}
	\rho_t'(ab)\rho_t(ab)^{-1} = \rho_t'(a) \rho_t(a)^{-1} +
		\rho_t(a) \rho_t'(b) \rho_t(b)^{-1} \rho_t(a)^{-1}.
\end{equation*}
The second term is defined to be
\begin{equation*}
	\Ad{\rho_t(a)}(\rho_t'(b) \rho_t(b)^{-1}) =
		\rho_t(a) \rho_t'(b) \rho_t(b)^{-1} \rho_t(a)^{-1}.
\end{equation*}

The Lie algebra of $G$, denoted by $\mathfrak{g}$, turns into a $\pi_1(M)$
module, with $\pi_1(M)$ acting via $\Ad{\rho_0}$. Then a cocycle of $\pi_1(M)$
with coefficients in $\mathfrak{g}$ twisted by $\Ad{\rho_0}$
 is defined as a map $z:\pi_1(M) \rightarrow \mathfrak{g}$, where
 $z(\gamma) = \rho'(\gamma)
\rho_0(\gamma)^{-1}$ and $\rho'$ is the derivative evaluated at $t=0$, such that
the map $z$ satisfies the cocycle condition
\begin{equation}
	z(ab) = z(a) + \Ad{\rho_0(a)} z(b).\label{eqn:cocycle}
\end{equation}
The group of all maps satisfying the cocycle condition in Equation
\eqref{eqn:cocycle} is defined to be $Z^1(\pi_1(M),\mathfrak{g}_{\Ad{\rho_0}})$.
Differentiating the triviality condition for representations
$\rho_t = g_t \rho_0 g_t^{-1}$ yields the coboundary condition
\begin{equation}
	z(\gamma) = u - \Ad{\rho_0(\gamma)}u \label{eqn:coboundary}
\end{equation}
for some $u \in \mathfrak{g}$. The set of cocycles satisfying Equation
\eqref{eqn:coboundary} are defined to be $B^1(\pi_1(M),\mathfrak{g}_{\Ad{\rho_0}})$, the
set of coboundaries of $\pi_1(M)$ with coefficients in $\mathfrak{g}$ twisted by
$\Ad{\rho_0}$.
Weil \cite{weil64, lubotzky85} has  noted that $Z^1(\pi_1(M),\mathfrak{g}_{\Ad{\rho_0}})$
contains the tangent space to $R(\pi_1(M),G)$ at $\rho_0$ as a subspace.
Provided that we can show that the representation variety
at $\rho_0$ is smooth, we can study the space of cocycles to
determine the first order behavior of deformations of a representation $\rho_0$.


\subsection{Hyperbolic geometry}

The hyperboloid model for $\mathbb{H}^3$ is described as a subspace of
$\mathbb{R}^{1,3}$.
Topologically, $\mathbb{R}^{1,3}$ is the space $\mathbb{R}^4$, but it is
endowed with the Lorentzian metric $ds^2 = -dx_1^2+dx_2^2+dx_3^2+dx_4^2$. Then,
\begin{equation*}
	\mathbb{H}^3 = \{\vec{x}=(x_1,x_2,x_3,x_4)\in \mathbb{R}^{1,3} :
		||\vec{x}||=-1, x_1>0\}
\end{equation*}
with the metric induced by $ds$ is
isometric to $\mathbb{H}^3$. The isometry group of $\mathbb{H}^3$ in
the hyperboloid model is the identity component $\SO^+(1,3)$ of
$\SO(1,3)$. Each point in the hyperboloid model
intersects exactly 1
line through the origin in $\mathbb{R}^{1,3}$. Hence, we can also identify
the hyperboloid with a subset of $\mathbb{R}P^3$, given by
\begin{equation*}
	\mathbb{H}^3 = \{[\vec{x}]=[x_1,x_2,x_3,x_4]\in \mathbb{R}P^3
		: ||\vec{x}||<0\}.
\end{equation*}

There is a well-known method for taking an isometry of $\mathbb{H}^3$ from
the upper
half-space model (i.e. an element $A \in \text{PSL}(2,\mathbb{C})$) to the
corresponding isometry in the hyperboloid model (see for instance
\cite[p. 66]{abbaspour07}). First, a point
$(x_1,x_2,x_3,x_4)$ from the hyperboloid model is identified with the matrix
\begin{equation*}
	P(x_1,x_2,x_3,x_4)=\begin{bmatrix}x_1+x_2 & x_3+ix_4\\
		x_3-ix_4& x_1-x_2\end{bmatrix}.
\end{equation*}
Then, $A$ acts on the point $(x_1,x_2,x_3,x_4)$ by
\begin{equation*}
	A P(x_1,x_2,x_3,x_4) A^*,
\end{equation*}
where $A^*$ denotes the Hermitian transpose of $A$. This operation
preserves $\det P = x_1^2-x_2^2-x_3^2-x_4^2$, so it sends points of
the hyperboloid in $\mathbb{R}^{1,3}$ to points of
the hyperboloid. The corresponding isometry
in the hyperboloid model is the element $A' \in \SO(1,3)$ so that
\begin{equation*}
	A P(x_1,x_2,x_3,x_4) A^* = P(A'(x_1,x_2,x_3,x_4)).
\end{equation*}

\subsection{Sol geometry}\label{sec:sol}

Topologically, Sol is $\mathbb{R}^3$, with the metric
$ds^2 = e^{2z}dx^2+e^{-2z}dy^2+dz^2$. In this model for $\Sol$, one can
see that by restricting to any plane $x=\text{constant}$,  we obtain a 2-dimensional
space that is isometric to the hyperbolic plane via the upper half-plane model.
Restricting to the plane $y=\text{constant}$ also yields a space isometric to
the hyperbolic plane as the lower half-plane model.

Sol also has an embedding into $\mathbb{R}P^3$ by
\begin{equation*}
	(x,y,z) \mapsto \begin{bmatrix}\cosh z\\ \sinh z \\ e^zx \\ e^{-z}y
		\end{bmatrix}.
\end{equation*}
The image of this map gives Sol as the subspace
\begin{equation*}
	\Sol = \{ [x_1,x_2,x_3,x_4] \in \mathbb{R}P^3: -x_1^2+x_2^2 < 0 \}.
\end{equation*}
The group $\PGL(4)$ contains the identity component of the isometry group of Sol
inside $\mathbb{R}P^3$ as elements of the form
\begin{equation*}
	\begin{bmatrix} \cosh c & \sinh c & 0 & 0\\
		\sinh c & \cosh c & 0 & 0\\
		a e^{c} & a e^{c} & 1 & 0\\
		b e^{-c} & -b e^{-c} & 0 & 1\end{bmatrix},
\end{equation*}
where $a,b,c \in \mathbb{R}$. Other components can be found by
multiplying the diagonal $2 \times 2$ blocks by $\pm 1$ or the upper
left $2\times 2 $ block by $\begin{bmatrix}0&1\\1&0\end{bmatrix}$.
A further treatment of
Sol geometry can be found in \cite{bonahon02}.

\subsection{HP geometry}\label{sec:HP}

There are also multiple copies of $\mathbb{H}^3$ lying inside
$\mathbb{R}^4$. For each $s > 0$, we can take the hyperboloid 
\begin{equation*}
	\mathbb{H}^3_s = \{ \vec{x} = (x_1,x_2,x_3,x_4) 
		: -x_1^2 + x_2^2 + x_3^2 + s^2x_4^2 = -1, x_1>0\},
\end{equation*}
and the subgroup $G_s$ of $\PGL(4,\mathbb{R})$ preserving the form
\begin{equation*}
	-x_1^2 + x_2^2 + x_3^2 + s^2 x_4^2,
\end{equation*}
to obtain a space isometric to $\mathbb{H}^3$. The isometry to the usual
hyperboloid model of $\mathbb{H}^3$ is given by the
rescaling map 
\begin{equation*}
	\mathfrak{r}_s = \begin{bmatrix} 1 & 0 & 0 & 0 \\ 0 & 1 & 0 & 0\\
		0 & 0 & 1 & 0 \\ 0 & 0 & 0 & s^{-1} \end{bmatrix}.
\end{equation*}

Geometrically, we can think of the family of hyperboloids, $\mathbb{H}^3_s$,
as flattening out to $\mathbb{H}^2 \times \mathbb{R}$ in $\mathbb{R}^4$.
Taking the limit as $s \rightarrow 0$ yields a model for half-pipe geometry.

Danciger \cite{danciger13} studies degenerations of singular
hyperbolic structures using the projective models. An appropriate
rescaling of the degeneration yields half-pipe (HP) geometry, a transition
geometry between hyperbolic geometry and anti-de Sitter (AdS) geometry.

Three-dimensional HP geometry, $\HP^3$, topologically is $\mathbb{R}^3$.
In terms of representations, it can be described as a rescaling of the
collapse of the structure group from $\SO(1,3)$ to $\SO(1,2)$.
Begin with a representation $\rho_1$ of $\pi_1(M)$ into $\SO(1,3)$, and
describe the collapse of the manifold in the $x_4$ coordinate by a family of
representations $\rho_t$, so that we end with a representation $\rho_0$ into
$\SO(1,2) \subset \SO(1,3)$ of matrices of the form
\begin{equation*}
	\rho_0(\gamma) = \begin{bmatrix} A \in \SO(1,2) & 0\\0 & 1\end{bmatrix}.
\end{equation*}
Conjugate the path of representations $\rho_t$ degenerating in this matter by
\begin{equation*}
	\mathfrak{r}(t) = 
	\begin{bmatrix} 1 & 0 & 0 & 0\\ 0 & 1 & 0 & 0\\ 0 & 0 & 1 & 0\\ 0 & 0 & 0 &
		t^{-1} \end{bmatrix},
\end{equation*}
and take the limit as $t \rightarrow 0$. This will yield a representation
$\rho_{\HP}$ whose image lies in the set of matrices of $\SO(1,3)$ of the
form
\begin{equation}
	\lim_{t \rightarrow 0} \mathfrak{r}(t) \rho_t(\gamma) \mathfrak{r}(t)^{-1} =
	\begin{bmatrix} A \in \SO(1,2) & 0 \\ \vec{v}^T & 1\end{bmatrix}
	=\rho_{\HP}(\gamma)
	\label{eqn:hpmatrix}
\end{equation}
where $\vec{v}^T$ is the transpose of a vector in $\mathbb{R}^3$. The vector
$\vec{v}$ can
be interpreted as an infinitesimal deformation of $A$ into $\SO(1,3)$.
A path of representations $\rho_t$ satisfying Equation \eqref{eqn:hpmatrix} is
said to be \textit{compatible to first order} with $\rho_\HP$.
The map $\mathfrak{r}(t)$ takes the standard
copy of $\mathbb{H}^3$ inside $\mathbb{R}^{1,3}$ to the isometric copy
$\mathbb{H}^3_t$. As we take the limit $t \rightarrow 0$, we obtain
$\HP^3$ as
\begin{equation*}
	\HP^3 = \lim_{t \rightarrow 0} \mathbb{H}^3_t = \{ (x_1,x_2,x_3,x_4)
		: -x_1^2 + x_2^2 +x_3^2 = -1, x_1>0 \}.
\end{equation*}
As a subset of $\mathbb{R}P^3$, we can think of $\HP^3$ as
\begin{equation*}
	\HP^3 = \{ [x_1,x_2,x_3,x_4]
		: -x_1^2 + x_2^2 +x_3^2 < 0 \}.
\end{equation*}
The structure group $G_\HP$ is the set of matrices of the form in Equation
\eqref{eqn:hpmatrix}.

A concrete description of $\vec{v}$ can be found by generalizing the isomorphism
$\SO(1,3) \cong \PSL(2,\mathbb{C})$. Let $\kappa_s$ be a non-zero
element such that $\kappa_s^2=-s^2$, and define an algebra $\mathcal{B}_s =
\mathbb{R} + \mathbb{R}\kappa_s$ generated over
$\mathbb{R}$ by $1$ and $\kappa_s$. Furthermore, define a conjugation by
\begin{equation*}
	a+b\kappa_s \mapsto \overline{a+b\kappa_s} = a-b \kappa_s.
\end{equation*}
Then let $A^*$ be the conjugate transpose of $A$.

We can define a map
$P_s = \mathbb{H}^3_s \subset \mathbb{R}^{1,3} \rightarrow \text{Herm}(
2,\mathcal{B}_s)$ by
\begin{equation*}
	P_s(x_1,x_2,x_3,x_4) = \begin{bmatrix} x_1 + x_2 & x_3 + \kappa_s x_4\\
		x_3 - \kappa_s x_4 & x_1 - x_2 \end{bmatrix}
\end{equation*}
where $\text{Herm}(2,\mathcal{B}_s)$ is the set of $2 \times 2$ matrices
with entries in $\mathcal{B}_s$ such that $A = A^*$. Then define the map
$\PSL(2,\mathcal{B}_s)\rightarrow G_s$ by $A \mapsto A'$ where
$A'$ is the matrix that satisfies
\begin{equation*}
	A P_s(x_1,x_2,x_3,x_4) A^* = P(A'(x_1,x_2,x_3,x_4)).
\end{equation*}
When $s=1$, this is the usual isometry from $\PSL(2,\mathbb{C})$
to $\SO(1,3)$. Danciger proves the following:

\begin{thm*}[\cite{danciger13}, Propositions 4.15, 4.19]
	For $s > 0$, the map $\PSL(2,\mathcal{B}_s) \rightarrow G_s$ is
	an isomorphism. When $s=0$, the map $\PSL(2,\mathcal{B}_0)
	\rightarrow G_0$ is an isomorphism onto the group of $\HP$ matrices.
\end{thm*}

Moreover, in the case $s=0$, we obtain a geometric interpretation for the
vector $\vec{v}$ in Equation \eqref{eqn:hpmatrix}. If we have a matrix in
$\PSL(2,\mathcal{B}_0)$, we can write it as $A+B \kappa_0$, where
$A$ is symmetric and $B$ is skew-symmetric. Similarly, we can write
$P_0(x_1,x_2,x_3,x_4) = X+Y\kappa_0$ where 
\begin{align*}
	X &= \begin{bmatrix}x_1+x_2&x_3\\x_3 & x_1-x_2\end{bmatrix}\\
	Y &= \begin{bmatrix}0&x_4\\-x_4&0\end{bmatrix}.
\end{align*}
Then $(A+B\kappa_0)(X+Y\kappa_0)(A+B\kappa_0)^* = AXA^T + (BXA^T - AXB^T +AYA^T)
\kappa_0$. In the map $\text{PSL}(2,\mathcal{B}_0)\rightarrow G_0$, the
symmetric part $AXA^T$ determines the first three rows of the HP matrix, and
the skew-symmetric part $(BXA^T - AXB^T + AYA^T)$ determines the bottom row of the
HP matrix.

\begin{lem*}[\cite{danciger13}, Lemma 4.20]
	Let $A+B\sigma$ have determinant $\pm1$. Then $\det A = \det(A+B\sigma)=\pm 1$
	and $\text{tr} BA^{-1} = 0$. In other words $B$ is in the tangent
	space at $A$ of matrices of constant determinant $\pm 1$.
\end{lem*}

Hence, when mapped into $\mathbb{R}P^3$, the symmetric part is the usual
map $\PSL(2,\mathbb{R}) \rightarrow \SO(1,2)$, and the bottom row of an HP
matrix comes from the skew-symmetric part. The vector $\vec{v}$ in the HP matrix of
Equation \eqref{eqn:hpmatrix} is an infinitesimal deformation of the $\SO(1,2)$
matrix from the collapsed structure.

The key result about $\HP$ structures is that we can recover hyperbolic
structures from them \cite[Proposition 3.6]{danciger13}.
Thus, if we can find an HP structure for $M_\phi$ and construct a transition at
the level of representations, then we can deform it to
nearby hyperbolic and AdS structures.

\section{The metabelian representation}\label{sec:metabelian}

Let $\phi:S \rightarrow S$ be a pseudo-Anosov homeomorphism with orientable
invariant foliations $\mathcal{F}^s, \mathcal{F}^u$ with singular set
$\sigma = \{s_0,s_1,\dots,s_n\}$
and transverse measures $\mu_s$ and $\mu_u$. If $S$ has a puncture $p_0$,
then we can fill in the puncture by taking $\bar{S} = S \cup \{p_0\}$. Either the
measured foliations extend smoothly to $p_0$, or $p_0$ is a singular point
of the foliation. In either case, we simply include $p_0$ in the set $\sigma$,
so we can simplify our analysis to the case where $S$ is closed.
The orientability assumption gives us some control over the eigenvalues of
$\phi^*:H^1(S) \rightarrow H^1(S)$. It also implies that the cone angles 
at the singular points in the singular Euclidean metric induced by the measured
foliations are multiples of $2\pi$ -- in particular, they are larger
than $2\pi$.

The following is a basic result about the eigenvalues of a pseudo-Anosov map,
see \cite{fathi79}, \cite{mcmullen03}, \cite{penner91}.

\begin{lem}[c.f. McMullen \cite{mcmullen03}, Theorem 5.3] \label{lem:eigenvalues}
	Let $\phi$ be a pseudo-Anosov homeomorphism with dilatation factor
	$\lambda$.
	Suppose also that $\phi$ has orientable unstable and stable foliations,
	$\mathcal{F}^u$ and $\mathcal{F}^s$.
	Then $\lambda$ and $\lambda^{-1}$ are simple eigenvalues of $\phi^*$.
\end{lem}

\begin{proof}
	If $\mathcal{F}^u$ and $\mathcal{F}^s$ are orientable, then their transverse
	measures $\mu_u, \mu_s$
	represent cohomology classes $\omega_\pm \in H^1(S)$. The fact
	that $\phi$ scales the invariant measures by $\lambda^{\pm 1}$
	implies that $\phi^* (\omega_\pm) = \lambda^{\pm 1} \omega_\pm$, so that
	$\lambda^{\pm 1}$ are eigenvalues of $\phi^*$.
	
	Let $\omega \in H^1(S)$ be any cohomology class dual to a simple closed
	curve $\gamma$. Since $\phi$ is pseudo-Anosov, $\phi^{\pm n}(\gamma)$ limits to
	the either $\mathcal{F}^u$ or $\mathcal{F}^s$. In particular,
	\begin{equation}\label{eqn:eigenvalue}
		\frac{(\phi^*)^{\pm n} \omega}{\lambda^{\pm n}} \rightarrow c \omega_\pm
	\end{equation}
	for some $c\neq 0$. Since the classes $\omega$ dual to simple closed curves
	span $H^1(S)$, the eigenspaces for $\lambda^{\pm 1}$ are 1-dimensional. In 
	fact, $\lambda^{\pm 1}$ must be simple eigenvalues by considering the Jordan
	canonical form. If there existed a generalized eigenvector $\omega$ such that
	$\phi^* \omega = \omega_\pm + \lambda^{\pm 1}\omega$, we would have
	$(\phi^*)^{\pm n}(\omega) = n \lambda^{\pm (n-1)} \omega_\pm
	+ \lambda^{\pm n} \omega$, so that the condition in Equation \eqref{eqn:eigenvalue}
	is not satisfied.
\end{proof}

Note that in addition to $\lambda$ and $\lambda^{-1}$ being simple eigenvalues,
we also have that the corresponding eigenvectors come from the measures
$\mathcal{F}^u$ and $\mathcal{F}^s$. In particular, if we take $\gamma_1,
\gamma_2,\dots,\gamma_{2g}$ to be a basis for $H_1(S)$, then the eigenvector
$\vec{e}_\lambda$ is given by
\begin{equation*}
	\vec{e}_\lambda = \begin{pmatrix} \mu_u (\gamma_1)\\
		\mu_u (\gamma_2)\\
		\vdots\\
		\mu_u(\gamma_{2g})\end{pmatrix},
\end{equation*}
where the transverse measure $\mu_u$ is taken to be a signed measure, i.e.
$\mu_u(-\gamma) = -\mu_u(\gamma)$, if $-\gamma$ is the closed curve
$\gamma$ taken with the orientation opposite that of $\mathcal{F}^u$. The
eigenvector corresponding to $\lambda^{-1}$ is given by
\begin{equation*}
	\vec{e}_{\lambda^{-1}} = \begin{pmatrix} \mu_s (\gamma_1)\\
		\mu_s (\gamma_2)\\
		\vdots\\
		\mu_s(\gamma_{2g})\end{pmatrix}.
\end{equation*}

Choose a disk $D$ that contains all of the points in $\sigma$, and fix a point
on $\partial D$ as the base point for $\pi_1(S\setminus \sigma)$.
Let $\delta_1,\delta_2,\dots,\delta_n$ be generators of
$\pi_1(S\setminus \sigma)$, so that each $\delta_i$ encircles exactly one
singularity $s_i$, each $\delta_i$ lies entirely inside $D$, and the
product $\delta_1 \delta_2 \cdots \delta_n$ is homotopic to the boundary
$\partial D$.

Choose standard generators $\alpha_1,\alpha_2,\dots,\alpha_{g}$ and
$\beta_1,\beta_2,\dots,\beta_g$ of $\pi_1(S)$ such that for each $i$,
 (a representative of) $\alpha_i$ and $\beta_i$ do not intersect
$\partial D$ for $i=1,\dots,g$, except at the basepoint for $\pi_1$.
We will also refer to these curves as
$\gamma_i = \alpha_i$, $\gamma_{g+i} = \beta_i$,
$\gamma_{2g+j} = \delta_j$.
When convenient, we will use $\alpha_i,\beta_i,$
and $\delta_j$ to refer to their respective homology classes.

On the dual generators $\alpha_i^*, \beta_i^*, \delta_j^*$ of
$H^1(S\setminus \sigma)$, $\phi^*$ has a block
upper triangular action: the first block on the diagonal
corresponding to the action on the closed
surface $S$, and the second block a permutation of the generators
$\delta_1^*,\dots,\delta_n^*$ coming from the curves around the singular
points. Strictly speaking,
this matrix is a square matrix with dimensions one greater than the dimension
of $H^1(S \setminus \sigma)$. There is one redundancy in the generators by the
relation $\sum_{j=1}^n \delta_j = 0$ in homology. However, using the additional
generator from the singularities makes the lower right block for $\phi^*$ easier
to understand. When discussing $H^1(S \setminus \sigma)$ (or $\phi^*$) in
this section, it will mean $H^1(S\setminus \sigma)$ with this additional
generator (resp. the action on $H^1(S \setminus \sigma)$ with the additional
generator).

Using these generators for $\pi_1(S\setminus \sigma)$, we can describe
$\Gamma = \pi_1(N_\phi=M_\phi \setminus \Sigma)$ by the following presentation.

\begin{align*}
	\Gamma = \left\langle \{\alpha_i\},\{\beta_i\},\{\delta_j\},\tau \left|
		\begin{array}{c}
			\tau \alpha_i \tau^{-1} = \phi(\alpha_i),
			\tau \beta_i \tau^{-1} = \phi(\beta_i),\\
			\tau \delta_j \tau^{-1} = w_j \delta_{k_j} w_j^{-1},
			\Pi_{i=1}^g [\alpha_i,\beta_i]=\Pi_{j=1}^n \delta_j
		\end{array}\right.
		\right\rangle,
\end{align*}
where $w_j$ are words in the $\alpha_i$s, $\beta_i$s, and $\delta_j$s.

We start with the metabelian representation $\rho_0:\Gamma \rightarrow
\PSL(2,\mathbb{R})$ with
\begin{equation*}
	\rho_0 (\gamma_i) = 
		\begin{bmatrix}
			1 & a_i= \mu_u(\gamma_i) \\
			0 & 1
		\end{bmatrix},
\end{equation*}
where $a_i$ is the signed length of $\gamma_i$ in $\mathcal{F}^u$. Note that
$a_i = 0$ for $2g < i \leq n$. We also set
\begin{equation*}
	\rho_0 (\tau) = 
		\begin{bmatrix}
			\sqrt{\lambda} & 0\\
			0 & \sqrt{\lambda}^{-1}
		\end{bmatrix},
\end{equation*}
where $\tau$ is the generator in the $S^1$ direction of $M_\phi$, and
$\lambda$ is the pseudo-Anosov dilatation factor of $\phi$.
There is a singular Sol structure on $M_\phi$ coming from the pseudo-Anosov
action on
$\mathcal{F}^u$ and $\mathcal{F}^s$, where $\mathcal{F}^u$ and $\mathcal{F}^s$
provide a singular Euclidean structure on the fibers of $M_\phi$. Recall from 
Section \ref{sec:sol} that Sol contains embedded hyperbolic planes as ``vertical''
planes. In the singular Sol structure on $M_\phi$, these can be seen as products of
a leaf of $\mathcal{F}^s$ with the $S^1$ direction. The metabelian representation
$\rho_0$ is a projection of the singular Sol structure along the leaves of $\mathcal{F}^u$
onto one of these hyperbolic
planes inside of Sol. Such a projection yields a \textit{transversely hyperbolic
foliation} -- locally, $M_\phi$ can be viewed as an open subset of
$\mathbb{H}^2 \times \mathbb{R}$, and the pseudometric is given by the metric on
the $\mathbb{H}^2$ factor and ignoring the second factor.


\section{Smoothness of the representation variety}\label{sec:smoothness}

The goal is to deform $\rho_0$ to a representation into
$\PSL(2,\mathbb{C})$, and to realize the representation as the holonomy
representation of a $(\mathbb{H}^3,\PSL(2,\mathbb{C}))$-structure on
$N$. We consider $\rho_0 \in R(\pi_1(N_\phi),\PSL(2,\mathbb{R}))$ as
the metabelian representation from the previous section.
We begin by computing the dimension of the space of classes of twisted cocycles
 $z \in H^1(\pi_1(N_\phi),
\mathfrak{sl}(2,\mathbb{C})_{\Ad{\rho_0}})$.

\begin{thm}\label{thm:representations}
	Let $\phi$ be pseudo-Anosov with stable and unstable foliations which are
	orientable. Suppose also that
	$\phi^*: H^1(S) \rightarrow H^1(S)$ does not have 1 as an
	eigenvalue. Then $\dim H^1(\Gamma, \mathfrak{sl}
	(2,\mathbb{C})_{\Ad{\rho_0}}) = k$ where $k$ is the number of components of
	the boundary of $N_\phi$.
\end{thm}

\begin{proof}
	Let $z \in Z^1(\pi_1(N_\phi),\mathfrak{sl}(2,\mathbb{C})_{\Ad{\rho_0}})$.
	Then $z$ is determined by its values on $\gamma_1,\dots,\gamma_{2g+n}$,
	and $\tau$, subject to the cocycle condition $\eqref{eqn:cocycle}$ imposed by
	the relations in $\Gamma$. These can be computed via the Fox calculus
	\cite[Chapter 3]{lubotzky85}. Differentiating the relations
	\begin{align*}
		\tau \gamma_i \tau^{-1} &= \phi(\gamma_i),
	\end{align*}
	yields
	\begin{align}
		\frac{\partial [ \phi(\gamma_i) \tau \gamma_i^{-1} \tau^{-1}]}{\partial \gamma_i}
			&= \frac{\partial\phi(\gamma_i)}{\partial \gamma_i}
				- \phi(\gamma_i)\tau\gamma_i^{-1}
				= \frac{\partial\phi(\gamma_i)}{\partial \gamma_i} - \tau\notag\\
		\frac{\partial [ \phi(\gamma_i) \tau \gamma_i^{-1} \tau^{-1}]}{\partial \gamma_j}
			&= \frac{\partial\phi(\gamma_i)}{\partial \gamma_j}, i \neq j\notag\\
		\frac{\partial [ \phi(\gamma_i) \tau \gamma_i^{-1} \tau^{-1}]}{\partial \tau}
			&= \phi(\gamma_i) - \phi(\gamma_i)\tau\gamma_i^{-1}\tau^{-1}
				= \phi(\gamma_i) - 1.\label{eqn:derivatives}
	\end{align}
	
	Choosing the basis,
	\begin{equation*}
		e_1 = \begin{bmatrix} 0 & 1 \\ 0 & 0\end{bmatrix},
		e_2 = \begin{bmatrix} 1 & 0\\ 0 & -1\end{bmatrix},
		e_3 = \begin{bmatrix} 0 & 0 \\ 1 & 0 \end{bmatrix}
	\end{equation*}
	for $\mathfrak{sl}(2,\mathbb{C})$, the values $z(\gamma_i)$ can be
	expressed in coordinates $(x_i,y_i,z_i)$, where $z(\gamma_i)$ is the matrix
	\begin{equation*}
		z(\gamma_i) = \begin{bmatrix} y_i & x_i \\z_i & -y_i\end{bmatrix},
	\end{equation*}
	and we similarly let $z(\tau)$ be given in the coordinates $(x_0,y_0,z_0)$. We note
	that by using the coboundary condition from Equation \eqref{eqn:coboundary},
	we can compute the set of coboundaries $B^1(\pi_1(N_\phi),\mathfrak{sl}
	(2,\mathbb{C})_{\Ad{\rho_0}})$ as the set of cocycle $z'$ satisfying,
	\begin{align*}
		z'(\gamma_i) &= \begin{bmatrix} -a_i z & 2a_iy+a_i^2 z\\
			0 & a_i z\end{bmatrix}\\
		z'(\tau) &= \begin{bmatrix} 0 & x-\lambda x\\
				z - \lambda^{-1} z & 0\end{bmatrix},
	\end{align*}
	where $x,y,z \in \mathbb{C}$ parametrize $B^1(\pi_1(N_\phi),\mathfrak{sl}
	(2,\mathbb{C})_{\Ad{\rho_0}})$. In particular, adding the appropriate coboundary
	$z'$ to $z$, we can set $x_0 = z_0 = 0$. To simplify the calculation somewhat,
	we will assume that $z(\tau)$ has this form
	\begin{equation*}
		z(\tau) = 
		\begin{bmatrix}
			y_0 & 0 \\
			0 & -y_0
		\end{bmatrix}.
	\end{equation*}
	
	We first note that if $W$ is a word in the $\gamma_i$, then
	$\rho(W) = \begin{bmatrix} 1 & A \\ 0 & 1\end{bmatrix}$ for some
	real number $A$. Then, under the chosen basis for $\mathfrak{sl}(2,
	\mathbb{C})$, $\Ad{\rho_0(W)}$ acts by
	\begin{equation*}
		\begin{bmatrix} 1 & -2A & -A^2\\ 0 & 1 & A \\ 0 & 0 & 1\end{bmatrix}.
	\end{equation*}
	We obtain one term from $\frac{\partial\phi(\gamma_i)}{\partial \gamma_j}$ for
	each instance of $\gamma_j$ in $\phi(\gamma_i)$ (with a negative sign if
	$\gamma_j^{-1}$ appears), and each term is a word in the $\gamma$'s.
	
	Similarly, we can compute that $\Ad{\rho_0(\tau)}$ acts on
	$\mathfrak{sl}(2,\mathbb{C})$ via
	\begin{equation*}
		\begin{bmatrix} \lambda & 0 & 0\\0 & 1 & 0\\ 0 & 0 & \lambda^{-1}
			\end{bmatrix}.
	\end{equation*}
	
	We see that $Z^1(\pi_1(N_\phi),\mathfrak{sl}(2,\mathbb{C})_{\Ad{\rho_0}})$
	is determined, as in \cite{heusener01}, by a subset of vectors
	$\vec{v}=(x_1,\dots,x_{2g+n},y_0,y_1,\dots,y_{2g+n},z_1,\dots,z_{2g+n})^T$
	such that $R\vec{v}=0$, where $R$ decomposes into blocks
	\begin{equation*}
		R = 
			\begin{bmatrix}
				\begin{bmatrix} & & \\ &\mathring\phi^*-\lambda I& \\ & & \end{bmatrix}
				& \begin{array}{c}
					-2 \lambda a_1 \\ \vdots \\-2 \lambda a_{2g+n}\end{array}
					& \begin{bmatrix} & \phantom{\mathring\phi^*-I} & \\ & K & \\ & & \end{bmatrix}
					& \begin{bmatrix} & \phantom{\mathring\phi^*-\lambda^{-1} I} & \\ & C  & \\
					& & \end{bmatrix}\\
				\begin{bmatrix} & \phantom{\mathring\phi^*-\lambda I} & \\ & 0 & \\ & &
					\end{bmatrix} &
					\begin{array}{c}
					0 \\ \vdots \\0 \end{array}
					& \begin{bmatrix} & & \\ & \mathring\phi^*-I & \\ & & \end{bmatrix}
					& \begin{bmatrix} & \phantom{\mathring\phi^*-\lambda^{-1} I} & \\ & D & \\
					& & \end{bmatrix}\\
				\begin{bmatrix} & \phantom{\mathring\phi^*-\lambda I} & \\ & 0 & \\ & &
					\end{bmatrix} & \begin{array}{c}
					0 \\ \vdots \\0 \end{array}
					& \begin{bmatrix} & \phantom{\mathring\phi^*-I} & \\ & 0 & \\ & & \end{bmatrix} &
					\begin{bmatrix} & & \\ & \mathring\phi^*-\lambda^{-1} I & \\ & & \end{bmatrix}
			\end{bmatrix}.
	\end{equation*}
	Here, $\mathring\phi^*:H^1(S\setminus \sigma)\rightarrow H^1(S\setminus \sigma)$
	is the $(2g+n) \times (2g+n)$ matrix describing the cohomology action induced
	by $\phi$, which can be written as a block matrix
	\begin{equation*}
		\begin{bmatrix}[\phi^*] & [*]\\ 0 & [P]\end{bmatrix}
	\end{equation*}
	where $P=(p_{ij})$ is a permutation matrix denoting the permutation of the
	singularities in $\sigma$ by $\phi$. In particular, if
	$\tau\delta_j\tau^{-1}=w_j \delta_{k_j} w_j^{-1}$, then $p_{jk_j}=1$.
	By Lemma \ref{lem:eigenvalues}, $\phi^*-\lambda I$ and $\phi^*-\lambda^{-1} I$ have
	1 dimensional kernel. Furthermore, since 1 is not an eigenvalue of $\phi^*$,
	$\mathring\phi^*-I$ has kernel whose dimension is
	equal to the number of disjoint cycles of the permutation of the punctures.
	But a cycle in the permutation corresponds to a single boundary
	component of $N_\phi$. Hence, the kernel of $R$ has dimension at most
	$2+k+1$, where the additional $1$ comes from the $(2g+n)+1$--th column
	of $R$ and
	\begin{equation*}
		k=\text{\# of components of }\Sigma=\text{\# of components of }\partial N.
	\end{equation*}
	Now consider the upper left portion of the matrix $R$, which we will call
	$U$:
	\begin{equation*}
		U = \left(\begin{array}{ccc}
			\begin{bmatrix} & & \\ & \mathring\phi^*-\lambda I & \\ & &\end{bmatrix}
				& \begin{array}{c}
					-2 \lambda a_1 \\ \vdots \\-2 \lambda a_{2g+n}\end{array}
				& \begin{bmatrix} & \phantom{\mathring\phi^*-I} &  \\ & K & \\
				& & \end{bmatrix}\\
			\begin{bmatrix} & \phantom{\mathring\phi^*-\lambda I} & \\ & 0 & \\
				& & \end{bmatrix}
				& \begin{array}{c}
					0 \\ \vdots \\0 \end{array}
				& \begin{bmatrix} & & \\ & \mathring\phi^*-I & \\ & & \end{bmatrix}
		\end{array}\right).
	\end{equation*}
	If $\text{null}(R) > 2+k$, then we must have that $\text{null}(U) > k+1$.
	
	Since $\lambda$ is a simple eigenvalue of $\phi^*$ and
	$(a_1,\dots,a_{2g})^T$ is a corresponding eigenvector for $\lambda$,
	$(a_1,\dots,a_{2g})^T$ is not in the image of $\phi^* - \lambda I$. Hence,
	for any $\vec{y}$ in the kernel of $\phi^* - I$, there is a unique $y_0$
	such that $K\vec{y}-y_0(a_1,\dots,a_{2g})^T$ is in the image of
	$\phi^*-\lambda I$. Therefore, $\text{null}(U) = k+1$
	
	Hence $\text{null}(R) = 2+k$. However, the solution arising from
	the kernel of $\mathring\phi^*-\lambda I$ is the eigenvector
	\begin{equation*}
		\vec{v}=(a_1,\dots,a_{2g+n},0,\dots,0,0,\dots,0)^T
	\end{equation*}
	which	is a coboundary. So we have that $\dim H^1(\Gamma, \mathfrak{sl}
	(2,\mathbb{R})_{\Ad{\rho_0}}) \leq k+1$. Finally, there is one further
	redundancy since 
	\begin{equation*}
		\Pi_{i=1}^g [\alpha_i,\beta_i]=\Pi_{j=1}^n \delta_j.
	\end{equation*}
	From the $\mathring\phi^*-I$ block, we can see that $y_{2g+1},\dots,y_{2g+n}$
	can be freely chosen as long as $y_{2g+j}=y_{2g+k_j}$ whenever
	$\tau \delta_j \tau^{-1} = w_j \delta_{k_j} w_j^{-1}$. Hence, the upper-left
	(=lower-right) entry of $z(\Pi_{j=1}^n \delta_j)$ can be freely chosen to
	be any quantity
	\begin{equation}
		y_{2g+1} + y_{2g+2} + \dots y_{2g+n}.\label{eqn:y-sum}
	\end{equation}
	The relation $\Pi_{i=1}^g [\alpha_i,\beta_i]=\Pi_{j=1}^n \delta_j$
	forces the sum in Equation \eqref{eqn:y-sum} to be a fixed quantity coming from
	the upper-left entry of $\Pi_{i=1}^g [\alpha_i,\beta_i]$, which has
	no dependence on $y_{2g+j}$, for $1\leq j \leq n$.
	
	Therefore, the relation drops the dimension of the space of cocycles by 1,
	and  $\dim H^1(\Gamma, \mathfrak{sl}(2,\mathbb{C})_{\Ad{\rho_0}})
	= k$.
\end{proof}

In order to show that $R(\pi_1(N_\phi),\PSL(2,\mathbb{C}))$ is smooth
at $\rho_0$, following \cite{heusener01, heusener05}, we define
a \textit{formal deformation} of $\rho: \pi_1(M) \rightarrow \PSL(2,\mathbb{C})$
for a fixed 3-manifold $M$
to be a homomorphism $\rho_\infty:\pi_1(M) \rightarrow \PSL(2,\mathbb{C}[[t]])$ of
the form
\begin{equation*}
	\rho_\infty(\gamma) = \pm \exp (\sum_{i=1}^\infty t^i u_i(\gamma))\rho(\gamma)
\end{equation*}
where $u_i : \pi_1(M) \rightarrow \mathfrak{sl}(2,\mathbb{C})$ are elements
of $C^1(\pi_1(M),\mathfrak{sl}(2,\mathbb{C})_{\Ad{\rho}})$, and evaluating
$\rho_\infty$ at $t=0$ yields $\rho$. If $\rho_\infty$ is a homomorphism 
modulo $t^{j+1}$, we say that $\rho_\infty$ is \textit{a formal deformation up to
order j}. A cocycle $u_1 \in Z^1(\pi_1(M),\mathfrak{sl}(2,\mathbb{C})_{\Ad\rho})$
is \textit{formally integrable} if there is a formal deformation of $\rho$ with leading
term $u_1$. In \cite{heusener01}, it is shown that given a deformation of order $j$,
there is an obstruction class $\zeta_{j+1} \in
H^2(\pi_1(M),\mathfrak{sl}(2,\mathbb{C})_{\Ad\rho})$ to extending to a deformation
of order $j+1$:

\begin{prop}[\cite{heusener01}, Proposition 3.1]
	Let $\rho \in R(\pi_1(M),\PSL(2,\mathbb{C}))$ and $u_i \in C^1(\pi_1(M),
	\mathfrak{sl}(2,\mathbb{C})_{\Ad\rho})$, $ 1 \leq i \leq j$ be given.
	If
	\begin{equation*}
		\rho_j(\gamma) = \exp(\sum_{i=1}^j t^i u_i(\gamma))\rho(\gamma)
	\end{equation*}
	is a homomorphism into $\PSL(2,\mathbb{C}[[t]])$ modulo
	$t^{j+1}$, then there exists an obstruction class $\zeta_{j+1}^{(u_1,\dots,u_k)} \in
	H^2(\pi_1(M),\mathfrak{sl}(2,\mathbb{C})_{\Ad\rho})$ such that:
	\begin{enumerate}
		\item There is a cochain $u_{j+1}:\pi_1(M) \rightarrow
			\mathfrak{sl}(2,\mathbb{C})$ such that
			\begin{equation*}
				\rho_{j+1}(\gamma) = \exp(\sum_{i=1}^{j+1} t^i u_i(\gamma)) 
					\rho(\gamma)
			\end{equation*}
			is a homomorphism modulo $t^{j+2}$ if and only if $\zeta_{j+1}=0$.
		\item The obstruction $\zeta_{j+1}$ is natural, i.e. if $f$ is a homomorphism
			then $f^* \rho_j := \rho_j \circ f$ is also a homomorphism modulo $t^{j+1}$
			and $f^*(\zeta_{j+1}^{(u_1,\dots,u_j)}) = \zeta_{j+1}^{(f^* u_1,\dots,f^* u_j)}$.
	\end{enumerate}
\end{prop}

We will denote by $i:\partial M \rightarrow M$ be the inclusion map.

\begin{lem}
	Let $M$ be a 3-manifold with torus boundary components
	$\partial M = \sqcup_{i=1}^k T_i$.
	Let $\rho:\pi_1(M) \rightarrow \PSL(2,\mathbb{C})$ be a non-abelian
	representation such that
	$\rho(\pi_1(T_i))$ contains a non-parabolic element for each component
	$T_i$ of $\partial M$. If $\dim H^1(
	\pi_1(M),\mathfrak{sl}(2,\mathbb{C})_{\Ad{\rho}}) = k$ where $k$ is the number
	of components of $\partial M$, then $i^*:H^2(M,\mathfrak{sl}(2,{\mathbb{C})_{
	\Ad\rho}}) \rightarrow H^2(\partial M,\mathfrak{sl}(2,{\mathbb{C})_{\Ad\rho}})$
	is injective.
\end{lem}

\begin{proof}
	We have the cohomology exact sequence for the pair $(M,\partial M)$
	\begin{equation*}
		\begin{CD}
			H^1(M,\partial M)@>>>
				H^1(M) @>\alpha>>
				H^1(\partial M) \\
				@>\beta>>
				H^2(M,\partial M) @>>>
				H^2(M) \\
				@>i^*>>
				H^2(\partial M) @>>>
				H^3(M,\partial M)@>>>
		\end{CD}
	\end{equation*}
	where all cohomology groups are taken to be with the twisted coefficients
	$\mathfrak{sl}(2,\mathbb{C})_{\Ad\rho}$.
	A standard Poincar\'{e} duality argument \cite{heusener01,hodgson98,porti97}
	gives that $\alpha$ has half-dimensional image. For a torus $T$,
	\begin{equation*}
		\dim H^1(\pi_1(T), \mathfrak{sl}(2,\mathbb{C})) = 2,
	\end{equation*}
	as long as $\rho(\pi_1(T))$ contains a
	hyperbolic element \cite{porti97}. Hence, $\alpha$ is injective. Since $\beta$ is
	dual to $\alpha$ under Poincar\'{e} duality, then $\beta$ is surjective. This
	implies that $i^*$ is injective.
\end{proof}

We apply the above to the representation $\rho_0$ to conclude that the
representation variety is smooth at the metabelian representation $\rho_0$.

\begin{thm}\label{thm:smoothness}
	The metabelian representation $\rho_0:\pi_1(N_\phi) \rightarrow \PSL(2,\mathbb{C})$
	is a smooth point of $R(\pi_1(N_\phi),\PSL(2,\mathbb{C}))$, with local dimension
	$k+3$.
\end{thm}

\begin{proof}
	We begin by showing that every cocyle in $Z^1(\pi_1(N_\phi),
	\mathfrak{sl}(2,\mathbb{C}))$ is integrable.
	
	Suppose we have $u_1,\dots,u_j : \pi_1(N_\phi) \rightarrow
	\mathfrak{sl}(2,\mathbb{C})$ such that 
	\begin{equation*}
		\rho_j(\gamma) = \exp(\sum_{i=1}^j t^i u_i(\gamma))\rho(\gamma)
	\end{equation*}
	is a homomorphism modulo $t^{j+1}$. We have that $\partial N_\phi =
	\sqcup_{i=1}^k T_i$ is a disjoint union of tori, and the restriction
	$\rho_j|_{\pi_1(T_i)}$ to $\pi_1(T_i)$ is also a formal deformation of order $j$.
	We have that $\rho_0 (T_i)$ contains a non-parabolic element, namely
	$\rho_0(\tau)$, or a translate. Then, the restriction of
	$\rho_0$ to $\pi_1(T_i)$ is a smooth point of the representation variety
	$R(\pi_1(T_i),\PSL(2,\mathbb{C}))$. Hence $\rho_j|_{\pi_1(T_i)}$ extends
	to a formal deformation of
	order $j+1$ by the formal implicit function theorem (see \cite{heusener01},
	Lemma 3.7). This implies that the restriction of $\zeta_{j+1}^{(u_1,\dots,u_j)}$
	to each component $H^2(T_i) < H^2(\partial N_\phi)$ vanishes.
	
	As $H^2(\partial N_\phi) = \oplus_{i=1}^k H^2(T_i)$,
	hence, $i^* \zeta_{k+1}^{(u_1,\dots,u_k)} = \zeta_{k+1}^{(i^*u_1,\dots,i^*u_k)} = 0$.
	We have shown in Theorem \ref{thm:representations} that $H^1(\pi_1(N_\phi),
	\PSL(2,\mathbb{C}))$ has dimension $k$.
	The injectivity of $i^*$ implies that $\zeta_{k+1}^{(u_1,\dots,u_k)} = 0$.
	
	Applying Proposition 3.6 from \cite{heusener01} to the formal deformation
	$\rho_\infty$ results in a convergent deformation. Hence, $\rho_0$ is a smooth
	point of the representation variety. As $\rho_0$ is non-abelian, we have
	that $\dim B^1(\pi_1(N_\phi),\PSL(2,\mathbb{C})) = 3$, so that
	$\dim R(\pi_1(N_\phi),\PSL(2,\mathbb{R})) = k+3$.
\end{proof}

\section{Singular hyperbolic structures}\label{sec:singularstructures}

In this section, we will use the smoothness result from Theorem
\ref{thm:smoothness} to find representations that are near
the Sol representation. In order to realize the representations as
geometric structures, we will need the Ehresmann--Thurston principle
\cite{thurston80}.

\begin{thm*}[Ehresmann--Thurston Principle]
	Let $X$ be a manifold upon which a Lie group $G$ acts transitively. Let $M$
	have a $(X,G)$-structure with holonomy representation $\rho:\pi_1(M)
	\rightarrow G$. For $\rho'$ sufficiently near $\rho$ in the space of
	representations  $\Hom(\pi_1(M),G)$, there exists a nearby
	$(X,G)$-structure on $M$ with holonomy representation $\rho'$.
\end{thm*}

To utilize the Ehresmann--Thurston principle, we will need to realize
all of our structure groups as subgroups of $\PGL(4,\mathbb{R})$.
We first study the process in which $\Sol$ can be seen as a limit of
$\HP = \HP^3$.

Given $s>0$, we let $\mathfrak{r}_1(s)$ be the rescaling map
\begin{equation*}
	\mathfrak{r}_1(s)=\begin{bmatrix}
		\frac{1}{2}(s+s^{-1}) & \frac{1}{2}(s-s^{-1}) & 0 & 0\\
		\frac{1}{2}(s-s^{-1}) & \frac{1}{2}(s+s^{-1}) & 0 & 0\\
		0 & 0 & 0 & -s\\
		0 & 0 & s^{-1} & 0
	\end{bmatrix}.
\end{equation*}
Then $\mathfrak{r}_1(s)$ take $\HP$ to
\begin{equation*}
	\HP_s = \{ [x_1,x_2,x_3,x_4] : -x_1^2+x_2^2 + s^2 x_4^2 < 0\},
\end{equation*}
which we think of as a copy of $\HP$ under a projective change of
coordinates. Conjugating $G_\HP$ by $\mathfrak{r}_1(s)$ gives the
structure group $G_{\HP_s}$ of $\HP_s$. Regular $\HP$ geometry is
given by the case $s=1$. Taking the limit as $s \rightarrow 0$ gives the
subset of $\mathbb{R}P^3$,
\begin{equation*}
	\HP_0 = \{ [x_1,x_2,x_3,x_4] : -x_1^2 + x_2^2 < 0\}.
\end{equation*}
Notice that this is exactly the image of the embedding of $\Sol$ into
$\mathbb{R}P^3$. We will use this fact to obtain a geometric transition
at the level of representations, and apply the Ehresmann--Thurston
Principle to obtain corresponding developing maps.
The map $\mathfrak{r}_1(s)$ can be thought of as the composition of
three maps: the first a hyperbolic translation by $\log s$, which causes
the $x_3$ and $x_4$ coordinates to converge to 0 in the projective
sense; followed by a rescaling to recover those coordinates;
and then a change of coordinates between $x_3$ and $x_4$ to obtain the
correct form for Sol. Hence, this can be thought of as a further
collapse onto a one-dimensional space, followed by a rescaling.
In order to insure
that the developing maps behave correctly, we will use the following lemma.

\begin{lem}[\cite{danciger13}, Lemma 3.7]\label{lem:developing_maps}
	Let $K$ be a compact set and let $F_t:K \rightarrow \mathbb{R}P^3$ be
	any continuous family of functions. Suppose $F_0(K)$ is contained
	in $\mathbb{X}_s$. Then there is an $\epsilon >0$ such that $|t|<\epsilon$
	and $|r-s| < \epsilon$ implies that $F_t(K)$ is contained in $\mathbb{X}_r$.
\end{lem}

Now, we can prove the following result.

\begin{thm} \label{thm:hyperbolicstructures}
	Let $\phi: S \rightarrow S$ be a pseudo-Anosov homeomorphism whose stable
	and unstable foliations, $\mathcal{F}^s$ and $\mathcal{F}^u$, are orientable
	and $\phi^*$ does not have 1 as an eigenvalue. Then, there exists a family
	of singular hyperbolic structures on $M_\phi$, smooth on the complement of
	$\Sigma$, that degenerate to a transversely hyperbolic foliation. The
	degeneration can
	be rescaled so that the path of rescaled structures limit to the singular
	Sol structure on $M_\phi$, as projective structures.
\end{thm}

\begin{proof}
	From the proof of Theorem \ref{thm:representations}, we can find a
	cocycle
	\begin{equation*}
		z\in Z^1(\pi_1(N_\phi),\mathfrak{sl}(2,\mathbb{R})_{\Ad{\rho_0}}))
	\end{equation*}
	corresponding to $\mathcal{F}^s$. The simple eigenvalue $\lambda^{-1}$
	of $\mathring\phi^*$ has corresponding eigenvector coming from $b_1=\mu_s(\gamma_1),
	\dots, b_{2g+n} = \mu_s(\gamma_{2g+n})$.
	More specifically, $\phi^*$ does not have $1$ as an eigenvalue, so we can solve
	\begin{equation}
		(\mathring\phi^*-I) \begin{pmatrix} y_1\\ \vdots\\y_{2g+n}\end{pmatrix} =
			-D\begin{pmatrix}b_1\\ \vdots\\b_{2g+n}\end{pmatrix},
			\label{eqn:yparameters}
	\end{equation}
	where $D_{2g\times 2g}$ is the restriction of $D$ to the upper left $2g \times
	2g$ entries.

	Finally, since $\lambda$ is a simple eigenvalue of $\mathring\phi^*$, we can
	also solve
	\begin{equation}
		(\mathring\phi^*-\lambda I) \begin{pmatrix}x_1\\ \vdots\\ x_{2g+n}\end{pmatrix}
			-2\lambda \begin{pmatrix}a_1\\ \vdots \\a_{2g+n}\end{pmatrix} y_0 =
			-K \begin{pmatrix}y_1\\ \vdots\\y_{2g+n}\end{pmatrix} -
			C \begin{pmatrix}b_1\\ \vdots\\b_{2g+n}\end{pmatrix}. \label{eqn:xparameters}
	\end{equation}

	Now we will use the above cocycle, which has the form:
	\begin{align*}
		z(\gamma_i) &= \begin{bmatrix} y_i & x_i \\ b_i & -y_i \end{bmatrix} \\
		z(\tau) &= \begin{bmatrix} y_0 & 0 \\ 0 & -y_0 \end{bmatrix}.
	\end{align*}
	The representation $\rho_0$ and the cocycle $z$ are converted into an
	$\HP$ representation, using the description of $G_\HP$ given in Section
	\ref{sec:HP}. In particular, $\rho_0$ and $z$ are combined to form
	a representation of $\pi_1(N_\phi)$ into $\PSL(2,\mathcal{B}_0)$ by
	$\gamma \mapsto \rho_0(\gamma) + z(\gamma)\rho_0(\gamma)\kappa_0$,
	then use the isomorphism from $\PSL(2,\mathcal{B}_0)$ to
	$G_0 = G_\HP$ to obtain:
	\begin{align*}
		\rho_{\HP}(\gamma_i)
		&= \begin{bmatrix}
			1+\frac{a_i^2}{2} & - \frac{a_i^2}{2} & a_i & 0\\
			\frac{a_i^2}{2} & 1-\frac{a_i^2}{2} & a_i& 0\\
			a_i & -a_i  & 1 & 0\\
			-b_i-a_i^2b_i + 2a_iy_i +x_i &
			-b_i+a_i^2b_i-2a_iy_i - x_i & 2y_i-2a_ib_i & 1
		\end{bmatrix} \\
		\rho_{\HP}(\tau) & = \begin{bmatrix}
			\frac{1}{2}(\lambda + \lambda^{-1}) & \frac{1}{2}(\lambda-\lambda^{-1})
				& 0 & 0\\
			\frac{1}{2}(\lambda - \lambda^{-1}) & \frac{1}{2}(\lambda+\lambda^{-1})
				& 0 & 0\\
			0 & 0 & 1 & 0\\
			0 & 0 &  2y_0 & 1 \end{bmatrix},
	\end{align*}
	Conjugating the $\HP$ representation by
	\begin{equation*}
		\mathfrak{r}_1(s)=\begin{bmatrix}
			\frac{1}{2}(s+s^{-1}) & \frac{1}{2}(s-s^{-1}) & 0 & 0\\
			\frac{1}{2}(s-s^{-1}) & \frac{1}{2}(s+s^{-1}) & 0 & 0\\
			0 & 0 & 0 & -s\\
			0 & 0 & s^{-1} & 0
		\end{bmatrix}
	\end{equation*}
	and taking $s \rightarrow 0$ gives the Sol representation
	\begin{align*}
		\rho_\text{Sol}(\gamma_i) &= \begin{bmatrix}
			1 & 0 & 0 & 0\\
			0 & 1 & 0 & 0\\
			b_i & b_i & 1 & 0\\
			a_i & -a_i & 0 & 1 \end{bmatrix}\\
		\rho_\text{Sol}(\tau) & = \begin{bmatrix}
			\frac{1}{2}(\lambda + \lambda^{-1}) & \frac{1}{2}(\lambda-\lambda^{-1}) & 0 & 0\\
			\frac{1}{2}(\lambda - \lambda^{-1}) & \frac{1}{2}(\lambda+\lambda^{-1}) & 0 & 0\\
			0 & 0 & 1 & 0\\
			0 & 0 & 0 & 1 \end{bmatrix}. 
	\end{align*}
	Thus, there is a family of $\HP$ representations that limit to the $\Sol$
	representation in $\text{PGL}(4,\mathbb{R})$, up to rescaling the path of $\HP$
	structures by $\mathfrak{r}_1(s)$.
	
	The structure groups for HP, $\mathbb{H}^3$, and Sol can
	be written as subgroups of $\text{PGL}(4,\mathbb{R})$, giving them
	$(\mathbb{R}P^3,\PGL(4,\mathbb{R}))$-structures.
	Since the Sol representation, as a representation into
	$\PGL(4,\mathbb{R})$, comes from an actual Sol structure on $N_\phi$, then
	by the Ehresmann--Thurston Principle, for small $s$,
	$\mathfrak{r}_1(s) \rho_\HP \mathfrak{r}_1(s)^{-1}$
	are holonomy representations for real projective structures, with developing
	maps $D_s$.
	
	Moreover, the $\Sol$ structure can be thought of as a $(\HP_0,G_{\HP_0})$
	structure, and applying Lemma \ref{lem:developing_maps} with
	$\mathbb{X} = \HP$ to $D_s$ and a compact fundamental domain for
	$N_\phi$, we see that for sufficiently small $s$, the projective structures from
	the Ehresmann--Thurston principle
	correspond to $\HP_s$ structures, which are rescaled $\HP$
	structures.

	Fix such an $s=s_0$, and consider the underlying $\HP$ structure.
	Since $R(\pi_1(N_\phi),\PSL(2,\mathbb{C}))$ is smooth at $\rho_0$,
	by work of Danciger \cite[Proposition 3.6]{danciger13}, there exists a 
	family of hyperbolic structures on $N_\phi$, given by their holonomy
	representations
	$\rho_t:\pi_1(N_\phi) \rightarrow \SO(1,3)$ such that at $t=0$, we
	obtain the $\SO(1,3)$ version of the representation $\rho_0$.
	Furthermore, conjugating $\rho_t$ by $\mathfrak{r}(t)$ yields
	$\rho_\HP$.
	
	For a fixed $s$,
	\begin{equation*}
		\mathfrak{r}_1(s)\mathfrak{r}(t)\rho_t\mathfrak{r}(t)^{-1}\mathfrak{r}_1(s)^{-1}
	\end{equation*}
	limits to $\mathfrak{r}_1(s)\rho_\HP\mathfrak{r}_1(s)^{-1}$. So taking the
	diagonal path
	\begin{equation*}
		\mathfrak{r}_1(t) \mathfrak{r}(t) \rho_t\mathfrak{r}(t)^{-1}\mathfrak{r}_1(t)^{-1}
	\end{equation*}
	yields a rescaling of $\rho_t$ that limits to the $\Sol$ structure.
\end{proof}

Note that the cocycle $z$ has the form:
\begin{equation*}
	z(\gamma_i) = \begin{bmatrix} y_i & x_i \\ b_i & -y_i\end{bmatrix}, 
\end{equation*}
where $b_i = \mu_s(\gamma_i)$. In particular, the deformation of $\rho_0$
contains the information of $\mathcal{F}^s$.
The deformation from the upper triangular representation $\rho_0$, which
is a projection parallel to $\mathcal{F}^u$ onto  a leaf of $\mathcal{F}^s$,
behaves like a deformation in a direction transverse to $\mathcal{F}^s$.

\section{Behavior of the singular locus}\label{sec:singularlocus}

Theorem \ref{thm:hyperbolicstructures} gives a family of hyperbolic structures
on $M_\phi \setminus \Sigma$. In general, the singular locus $\Sigma$ may not
remain as cone singularities. In this section, we will show that it is possible
to control the singularities so that we obtain a family of nearby cone
manifolds.

The manifold $N_\phi = M_\phi \setminus \Sigma$ has torus
boundary components, $\partial N_\phi = \sqcup_{i=1}^k{T_i}$. Let $m_i$ be a
meridian curve for $T_i$, and $l_i$ a longitudinal curve. There is
a model for a torus $T$ degenerating to the HP
structure described by the representation
\begin{align*}
	\rho_{\HP} (m) &= \begin{bmatrix} 1&0&0&0 \\ 0&1&0&0 \\ 0&0&1&0\\ 0&0&\omega&1
		\end{bmatrix},\\
	\rho_{\HP} (l) &= \begin{bmatrix} \cosh d & \sinh d & 0 & 0\\
		\sinh d & \cosh d & 0 & 0\\
		0 & 0 & \pm 1 & 0 \\ 0 & 0 & \mu & \pm 1 \end{bmatrix},
\end{align*}
which is given in \cite{danciger13}.
In particular, take the family of representations into
$\SO(1,3)$ such that
\begin{align*}
	\rho_t(m) &= \begin{bmatrix} 1&0&0&0 \\ 0&1&0&0 \\
		0 & 0 & \cos \omega t & -\sin \omega t\\
		0 & 0 & \sin \omega t & \cos \omega t \end{bmatrix},\\
	\rho_t(l) &= \begin{bmatrix} \cosh d & \sinh d & 0 & 0\\
		\sinh d & \cosh d & 0 & 0 \\
		0 & 0 & \pm \cos \mu t & - \sin \mu t\\
		0 & 0 & \sin \mu t & \pm \cos \mu t\end{bmatrix}. 
\end{align*}
Then, conjugating by
\begin{equation*}
	\begin{bmatrix} 1&0&0&0 \\ 0&1&0&0 \\ 0&0&1&0 \\ 0&0&0&t^{-1}\end{bmatrix}
\end{equation*}
and taking the limit as $t\rightarrow 0$ yields $\rho_{\HP}(m)$ and
$\rho_{\HP}(l)$. Thus, $\omega$,
which is called the \textit{infinitesimal rotation} in \cite{danciger13}, describes the
infinitesimal change in the cone angle about that component of the singularity.

In the case that $\Sigma$ has multiple components, as in our case, we can
modify the computation. From the construction of $\rho_0$, we can see that
each $\rho_0(l_i)$ is a hyperbolic translation with an axis in $\mathbb{H}^2$
having a common endpoint at infinity. Specifically, they all differ from
\begin{equation*}
	\rho_0(\tau) = \begin{bmatrix} \sqrt{\lambda} & 0 \\ 0 & \sqrt{\lambda}^{-1}
	\end{bmatrix}
\end{equation*}
by a parabolic element. Namely, there exists some parabolic of the form
\begin{equation*}
	\begin{bmatrix} 1 & a \\ 0 & 1\end{bmatrix} \in \PSL(2,\mathbb{R}),
\end{equation*}
taking $\rho_0 (\tau)$ to $\rho_0(l_i)$.
If this is deformed by the infinitesimal isometry
\begin{equation*}
	\begin{bmatrix} y & x \\ b & -y \end{bmatrix} \in \mathfrak{sl}(2,\mathbb{R}),
\end{equation*}
the deformation is encapsulated by the $\HP$ matrix,
\begin{equation*}
		\begin{bmatrix}
			1+\frac{a^2}{2} & - \frac{a^2}{2}& a& 0\\
			\frac{a^2}{2} & 1-\frac{a^2}{2} & a& 0\\
			a & -a& 1 & 0\\
			- b -a^2b + 2ay+ x &
			- b +a^2b-2ay- x & 2y-2ab & 1
		\end{bmatrix},
\end{equation*}
which is the $\PGL(4,\mathbb{R})$ form of the $\PSL(2,\mathcal{B}_0)$ element,
\begin{equation*}
	\begin{bmatrix} 1&a\\0&1\end{bmatrix} +
		\begin{bmatrix} y & x\\b&-y \end{bmatrix}\begin{bmatrix} 1&a\\0&1\end{bmatrix} 
		\kappa_0.
\end{equation*}

Then, for a general singularity, the representation $\rho_{\HP}$ should be such
that $\rho_{\HP}(m_i)$ and $\rho_{\HP}(l_i)$ are conjugates of $\rho_{\HP}(m)$
and $\rho_{\HP}(l)$, with the conjugating matrix being of the above type.
This gives the general form:
\begin{align}
	\rho_{\HP}(m_i) &=
		\begin{bmatrix} 1&0&0&0 \\ 0&1&0&0 \\ 0&0&1&0\\
		-a\omega &a\omega &\omega&1 \end{bmatrix},\notag\\
	\rho_{\HP}(l_i) &= \begin{bmatrix}
		\mp a^2 + (1+a^2)C & \pm a^2 - a^2C + S & a(1-e^d) & 0\\
		\mp a^2 + a^2C + S & \pm a^2 + C -a^2C & a(1-e^d) & 0\\
		a(\mp 1 + e^{-d}) & a(\pm 1 - e^{-d}) & \pm 1 & 0\\
		f_1 & f_2 & 2ab(e^d \mp 1) + \mu & \pm 1\\
		\end{bmatrix}
			\label{eqn:translated_cone}
\end{align}
where
\begin{align*}
	C &= \cosh d,\\
	S &= \sinh d,\\
	f_1&=-a \mu - (b+2a^2b-2ay)(e^d \pm 1) + x(e^{-d}\mp 1),\\
	f_2&=a\mu + (2a^2b-2ay-b)(e^d \mp 1) -x( e^{-d} \mp 1).
\end{align*}

The curves $\delta_j=\gamma_{2g+j}$ are meridians of the boundary tori, so we
verify that $\rho_{\HP}(\delta_j)$ agrees with the description of
$\rho_{\HP}(m_i)$. From our computation of $\rho_{\HP}(\gamma_{2g+j})$, we
notice that $a_{2g+j} = b_{2g+j} = 0$ since the signed length of $\delta_j$
around any singular point of the foliation is $0$, so
\begin{equation*}
	\rho_{\HP}(\delta_j) =
		\begin{bmatrix} 1 & 0 & 0 & 0\\ 0 & 1 & 0 & 0\\
			0 & 0 & 1 & 0 \\ x_{2g+j} & -x_{2g+j} & 2y_{2g+j} & 1
		\end{bmatrix}.
\end{equation*}
Hence, the infinitesimal rotation is given by $\omega=2y_{2g+j}$, where
the $y_{2g+j}$ can be chosen
freely as long as they are the same for singular points in the same orbit
of $\phi$. It remains to show that $x_{2g+j} = -a\omega = -2ay_{2g+j}$,
where $a$ is the amount of parabolic translation that takes the axis between
$0$ and infinity to the axis given by the orbit of the singular point $s_j$.

Suppose that $m$ is the order of the orbit of singular points that contains
the singularity encircled by $\delta_j$. Then,
$\phi^m(\delta_j) = v_j \delta_j v_j^{-1}$ for some word $v_j \in \pi_1(S
\setminus \sigma)$.
Noting that $\rho_0(\delta_j) = \begin{bmatrix} 1 & 0 \\ 0 & 1 \end{bmatrix}$
and $\rho_0(v_j) = \begin{bmatrix}1 & A \\ 0 & 1\end{bmatrix}$ for some
real number $A$,
the twisted cocycle condition yields, by using
Equation \eqref{eqn:derivatives} with $\gamma_{2g+j} = \delta_j$, that
\begin{equation}
	\begin{bmatrix} y_{2g+j} & \lambda^m x_{2g+j}\\
		0 & -y_{2g+j} \end{bmatrix}
	= \begin{bmatrix} y_{2g+j} & x_{2g+j} - 2y_{2g+j}A\\
		0 & -y_{2g+j}\end{bmatrix}. \label{eqn:boundary_cocycle}
\end{equation}
This follows because $b_{2g+j}=0$.

In addition to $\tau^m \delta_j \tau^{-m} = \phi^m(\delta_j) = v_j \delta_j v_j^{-1}$,
we have that $\tau^m l_j \tau^{-m} = v_j l_j v_j^{-1}$. As previously noted,
$\rho_0(l_j)$ is conjugate to $\rho_0(\tau)^m$ by the parabolic element
$\begin{bmatrix} 1 & a\\0&1\end{bmatrix}$. This yields
\begin{equation*}
	\rho_0(l_j)=\begin{bmatrix}\sqrt{\lambda}^m & a(-\sqrt{\lambda}^m +
		\sqrt{\lambda}^{-m})\\ 0 & \sqrt{\lambda}^{-m}\end{bmatrix}.
\end{equation*}
From the relation $\tau^m l_j \tau^{-m} = v_j l_j v_j^{-1}$, we obtain that
\begin{equation*}
	\begin{bmatrix}\sqrt{\lambda}^m & a\lambda^m(-\sqrt{\lambda}^m +
		\sqrt{\lambda}^{-m})\\ 0 & \sqrt{\lambda}^{-m}\end{bmatrix}
		= \begin{bmatrix}\sqrt{\lambda}^m & (A+a)(-\sqrt{\lambda}^m +
		\sqrt{\lambda}^{-m})\\ 0 & \sqrt{\lambda}^{-m}\end{bmatrix}.
\end{equation*}
This yields $A=\lambda^m a - a = a(\lambda^m -1)$. The cocycle
condition from Equation \eqref{eqn:boundary_cocycle} yields
$x_{2g+j}(\lambda^m-1) = -2y_{2g+j}a(\lambda^m -1)$, which is
exactly the desired condition $x_{2g+j} = -2ay_{2g+j}$. A similar computation
can be used to find the parameters $x$ and $b$,
with $b$ equaling the $\mu_s$ distance between $\tau$ and $l_i$.
The longitudinal curves $\rho_{\HP}(l_i)$ are conjugates of multiples of
$\rho_{\HP}(\tau)$. Since $\rho_{\HP}(\tau)$ has the form stipulated in Equation
\eqref{eqn:translated_cone} for $\rho_{\HP}(l_i)$, we have first order
compatibility of the $\HP$ representation with representations of
cone singularities.
From the previous computation of $\rho_{\HP}(\tau)$, we can see that
$d = m\log \lambda$ and $\mu=2my_0$.

In order to show that the components of the singular locus remain as cone
singularities, we will additionally need to show that the 
subset of structures where the meridian curves remain elliptic is smooth so
that the first order compatibility can be realized by a path of structures
on $N_\phi$. The proof generalizes
\cite[Lemma 4.25]{danciger13} to multiple components.

\begin{lem}[c.f. \cite{danciger13}, Lemma 4.25]
	\label{lem:conesingularities}
	The subset of $H^1(\pi_1 (N_\phi),\mathfrak{sl}(2,\mathbb{C})_{\Ad{\rho_0}})$
	corresponding to singular hyperbolic structures near $\rho_0$ such that
	$\rho_t(m_i)$ remains elliptic has real dimension $k$.
\end{lem}

\begin{proof}
	The complex dimension of $H^1(\pi_1(T_i),
	\mathfrak{sl}(2,\mathbb{C})_{\Ad{\rho_0}})$, where
	$T_i$ is a boundary component homeomorphic to a torus, is $2$, given by
	the differentials $dl(l_i)$ and $dl(m_i)$ of the lengths $l(l_i)$ and $l(m_i)$.
	The subspace of $H^1(\pi_1(T_i),\mathfrak{sl}(2,\mathbb{C})_{\Ad{\rho_0}})$ where
	$\rho(m_i)$
	remains elliptic as it is deformed by a cocycle has real dimension equal to 3.
	
	By the Poincar\'{e} duality argument, the image of
	$H^1(\pi_1(N_\phi),\mathfrak{sl}(2,\mathbb{C})_{\Ad{\rho_0}})$ in
	$H^1(\pi_1(T_i),\mathfrak{sl}(2,\mathbb{C})_{\Ad{\rho_0}})$
	has real dimension 2, for each torus component $T_i$.	
	Moreover, from the 
	computation of the space of
	cocycles, we can pick $z \in H^1(\pi_1(N_\phi),
	\mathfrak{sl}(2,\mathbb{C})_{\Ad{\rho_0}})$
	with $y_{2g+i}$ arbitrarily large, so that $z(m_i)$ increase translation length.
	Thus, the image is transverse to the subset
	of $H^1(\pi_1(T_i),\mathfrak{sl}(2,\mathbb{C})_{\Ad{\rho_0}})$ where
	$\rho(m_i)$ remains elliptic.
	Noting that $\partial N_\phi$ is a disjoint union $\sqcup T_i$, the
	subset of the image of $H^1(\pi_1(N_\phi),\mathfrak{sl}(2,\mathbb{C})_{\Ad{\rho_0}})$
	in $H^1(\pi_1(\partial N_\phi),\mathfrak{sl}(2,\mathbb{C})_{\Ad{\rho_0}})$ has real
	dimension $k$.
\end{proof}

Lemma \ref{lem:conesingularities}, along with Theorem
\ref{thm:hyperbolicstructures}, tells us that the we can choose a family of
hyperbolic structures on $N_\phi$ near
the $\Sol$ structure on $N_\phi$ such that the restriction of the corresponding
representations to the boundary tori agree with representations of the
models for cone singularities. After a finite number of applications of
Proposition 4.3 and Proposition 4.10 from \cite{danciger13}, once on each
component of $\Sigma$, we conclude that the representations
can be realized as actual hyperbolic cone structures. We restate those
propositions here.

\begin{prop*}[\cite{danciger13}, Proposition 4.3]
	Let $M$ be a manifold with a projective structure on $N=M \setminus \Sigma$
	with cone-like singularities along $\Sigma=\{\gamma\}$. Let $B$ be a small neighborhood
	of a point $p \in \Sigma$, with $\Sigma_B = \Sigma \cap B$. Then:
	\begin{enumerate}
		\item The developing map $D$ on $\widetilde{B \setminus \Sigma_B}$
			extends to the universal branched cover $\tilde{B} =
			\widetilde{B \setminus \Sigma_B} \cup \Sigma_B$ of $B$ branched
			over $\Sigma_B$.
		\item $D$ maps $\Sigma_B$ diffeomorphically onto an interval of a line
			$\mathfrak{L}$ in $\mathbb{R}P^3$.
		\item The holonomy $\rho(\pi_1(B \setminus \Sigma_B))$ point-wise
			fixes $\mathfrak{L}$.
	\end{enumerate}
\end{prop*}

\begin{prop*}[\cite{danciger13}, Proposition 4.10]
	Suppose $\rho_t : \pi_1(M) \rightarrow \PGL(4,\mathbb{R})$ is a path of
	representations such that:
	\begin{enumerate}
		\item $\rho_0$ is the holonomy representation of a projective structure
			on $N = M \setminus \Sigma$ with cone-like singularities
			along $\Sigma=\{\gamma\}$, and $\mathfrak{L}$ is the line in
			$\mathbb{R}P^3$ fixed by $\rho_0(\pi_1 (\partial M))$.
		\item $\rho_t(m)$ point-wise fixes a line $\mathfrak{L}_t$ with
			$\mathfrak{L}_t \rightarrow \mathfrak{L}$.
	\end{enumerate}
	Then, for all $t$ sufficient small, $\rho_t$ is the holonomy representation
	for a projective structure on $N$ with cone-like singularities along
	$\Sigma$.
\end{prop*}

A computation of the commutator $\rho_{\HP}([\alpha_i,\beta_i])$ yields a
matrix of the form in Equation \eqref{eqn:translated_cone}:
\begin{equation*}
	\begin{bmatrix} 1 & 0 & 0 & 0 \\ 0 & 1 & 0 & 0\\ 0 & 0 & 1 & 0\\
		- f &  f & g & 1 \end{bmatrix},
\end{equation*}
where
\begin{align*}
	f&=a_{g+i}^2b_i+2a_{g+i}y_i-a_i^2b_{g+i}-2a_iy_{g+i}\\
	g&= -2a_{g+i}b_i+2a_ib_{g+i}.
\end{align*}

Therefore, the product
of the commutators $\rho_{\HP}(\Pi_{i=1}^g [\alpha_i,\beta_i])$ also has
this form. In the case where $\gamma_{2g+j} = \delta_j$, we also have that
\begin{equation*}
	\rho_{\HP}(\delta_j) = \rho_{\HP}(\gamma_{2g+j}) =
		\begin{bmatrix}1 & 0 & 0& 0\\0 & 1 & 0 & 0\\ 0 & 0 & 1 & 0\\
			x_{2g+j} & -x_{2g+j} & 2y_{2g+j} & 1\end{bmatrix}.
\end{equation*}
Note that $y_{2g+j} = y_{2g+j'}$ if $\delta_j$ and $\delta_{j'}$ belong in the
same cycle of the permutation (i.e. they are meridians for the same component
of $\Sigma$). In other words, we have cone-type singularities that develop in
the singular hyperbolic structure, and for each component of $\Sigma$, there is
freedom in choosing the infinitesimal cone angle about that component.
Moreover, the commutator/singularities relation
\begin{equation*}
	\prod_{i=1}^g [\alpha_i,\beta_i]=\prod_{j=1}^n \delta_j
\end{equation*}
says that the sum of the infinitesimal cone angles about each component,
weighted by the number of singularities in the permutation for that component,
must equal some quantity $\omega_{tot}$ determined by the loop
$\prod_{i=1}^g [\alpha_i,\beta_i]$ that encircles all of the singularities.

\begin{lem}\label{lem:coneangle}
	The total infinitesimal cone angle $\omega_{tot}$ is non-zero.
\end{lem}

\begin{proof}
	A straight-forward computation shows that the $\omega=\omega_{tot}$ entry in
	the commutator $\rho_{\HP}([\alpha_i,\beta_i])$ is given by
	$2(a_ib_{g+i}-a_{g+i}b_i)$. Hence, the $\omega$ entry in the product
	\begin{equation*}
		\rho_{\HP}(\Pi_{i=1}^g [\alpha_i,\beta_i])
	\end{equation*}
	is the negative of algebraic intersection pairing
	$\hat{i}(\vec{e}_\lambda,\vec{e}_{\lambda^{-1}})$. We note that the
	algebraic intersection is a symplectic form on $H^1(S)$.

	Suppose $e_\mu$ is an
	eigenvector of $\phi^*$ with eigenvalue $\mu \neq \lambda$. Then
	\begin{equation*}
		\hat{i}(\vec{e}_\mu,\vec{e}_{\lambda^{-1}}) =
			\hat{i}(\phi^* \vec{e}_\mu, \phi^* \vec{e}_{\lambda^{-1}})
			 = \mu \lambda^{-1} \hat{i}(\vec{e}_\mu, \vec{e}_{\lambda^{-1}}).
	\end{equation*}
	Since $\mu \neq \lambda$, this means that $\hat{i}(\vec{e}_\mu,
	\vec{e}_{\lambda^{-1}}) = 0$.
	
	If $\vec{e}_{\mu,p}$ is a generalized eigenvector such that $(\phi^* -\mu I)^p
	\vec{e}_{\mu,p} = 0$, then we induct on $p$. Notice that $\phi^*
	\vec{e}_{\mu,p} = \mu \vec{e}_{\mu,p} + c \vec{e}_{\mu,p-1}$, where
	$(\phi^* -\mu I)^{p-1} \vec{e}_{\mu,p-1} = 0$. Hence, if
	$\hat{i}(\vec{e}_{\mu,p-1},\vec{e}_{\lambda^{-1}}) = 0$, then
	it must be that $\hat{i}(\vec{e}_{\mu,p},\vec{e}_{\lambda^{-1}}) = 0$ as well since
	\begin{equation*}
		\hat{i}(\vec{e}_{\mu,p},\vec{e}_{\lambda^{-1}}) =
			\hat{i}(\phi^* \vec{e}_{\mu,p},\phi^* \vec{e}_{\lambda^{-1}})
			= \mu \lambda^{-1} \hat{i}(\vec{e}_{\mu,p},\vec{e}_{\lambda^{-1}}).
	\end{equation*}
	The generalized eigenvectors
	of $\phi^*$ span $\mathbb{R}^{2g}$ and $\lambda$ is a simple eigenvalue,
	so that means that if $\hat{i}(\vec{e}_\lambda,\vec{e}_{\lambda^{-1}}) = 0$,
	then $\hat{i}(\vec{u},\vec{e}_{\lambda^{-1}}) = 0$ for all
	$\vec{u} \in \mathbb{R}^{2g}$, contradicting
	the non-degenerate condition for symplectic forms.
\end{proof}

We can now prove Theorem \ref{thm:conestructures}.

\begin{thm}\label{thm:conestructures}
	Let $\phi: S \rightarrow S$ be a pseudo-Anosov homeomorphism whose stable
	and unstable foliations, $\mathcal{F}^s$ and $\mathcal{F}^u$, are orientable
	and $\phi^*:H^1(\bar{S}) \rightarrow H^1(\bar{S})$ does not have 1 as an
	eigenvalue. Then, there exists a family
	of singular hyperbolic structures on $M_\phi$, smooth on the complement of
	$\Sigma$ and with cone singularities along $\Sigma$, that degenerate to a
	transversely hyperbolic foliation. The degeneration can
	be rescaled so that the path of rescaled structures limit to the singular
	Sol structure on $M_\phi$, as projective structures. Moreover, the cone
	angles can be chosen to be decreasing.
\end{thm}

\begin{proof}
	Lemma \ref{lem:conesingularities} and Theorem \ref{thm:hyperbolicstructures}
	imply that there exist a family of hyperbolic structures on $N_\phi$ near the
	$\Sol$ structure on $N_\phi$ such that the meridian and longitudinal curves
	of the boundary tori have the form in Equation \eqref{eqn:translated_cone}.
	
	Apply Proposition 4.3 and Proposition 4.10 from \cite{danciger13} on 
	one component $\gamma$ of $\Sigma$, to show that
	$M_\phi \setminus (\Sigma \setminus \gamma)$ has a projective
	structure with holonomy $\rho_t$ with cone-like singularities along
	$\gamma$ for sufficiently small $t$. Proceed inductively on each
	component of $\Sigma$. 
	
	Lemma \ref{lem:coneangle} implies that the infinitesimal cone angles of
	each boundary component can be chosen to be negative, so that the cone
	angles are all decreasing. The total
	infinitesimal cone angle $\omega_{tot} \neq 0$, and the proof of Lemma
	\ref{lem:coneangle} shows that it is the negative of $\hat{i}(\vec{e}_\lambda,
	\vec{e}_{\lambda^{-1}})$, and taking a positive orientation for $\{\vec{e}_\lambda,
	\vec{e}_{\lambda^{-1}} \}$ leads to $\omega_{tot} < 0$.
\end{proof}

As a remark, the results of \cite{danciger13} also imply that there
are nearby $\AdS$ structures that collapse to the same transversely
hyperbolic foliation, such that a similar rescaling gives the $\HP$ structure.
The generalizations made here to those results can also easily be made for
$\AdS$ structures, so there are also nearby $\AdS$ structures with 
tachyon (cone-like) singularities.

\section{Genus 2 Example}

We will compute the representations and parameters to find the deformation
in a genus two example. Begin with the curves $\alpha_1, \alpha_2,
\beta_1,\beta_2$, which form the symplectic basis for $H_1(S)$. We begin with
left Dehn twists $T_{\beta_1}, T_{\beta_2}, T_{\gamma}$ along $\beta_1,
\beta_2$, and $\gamma$, followed by right Dehn twists $T_{\alpha_1}^{-1},
T_{\alpha_2}^{-1}$ along $\alpha_1$ and $\alpha_2$. Since the disjoint sets
of curves $\{\alpha_1,\alpha_2\}$ and $\{\beta_1,\beta_2,\gamma\}$ fill, the
resulting homeomorphism $\phi:S\rightarrow S$ is pseudo-Anosov (see
\cite{penner88} or \cite[p. 398]{farb12}).

\begin{figure}
	\begin{center}
		\includegraphics{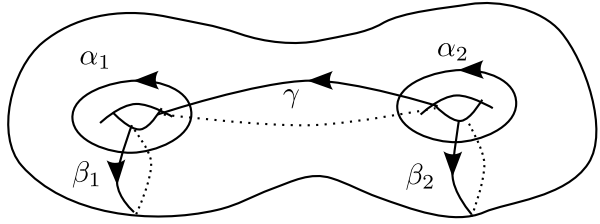}
		\caption{The curves $\alpha_1,\alpha_2,\beta_1,\beta_2$ which form the
			basis for $H_1(S)$, and $\gamma$.}
		\label{fig:genus2}
	\end{center}
\end{figure}

The stable and unstable foliations are orientable with two singular points
of cone angle $4\pi$, one in each of the two components of $S \setminus
\{\alpha_1,\alpha_2,\beta_1,\beta_2,\gamma\}$. A train track for
$\mathcal{F}^u$ is shown in Figure \ref{fig:genus2_train_track}, and we can
verify that the foliations are orientable with two singularities $s_1$ and
$s_2$.

\begin{figure}
	\begin{center}
		\includegraphics{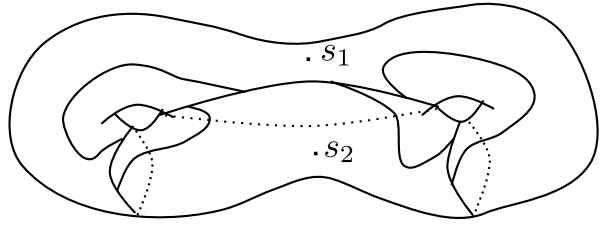}
		\caption{A train track for $\mathcal{F}^u$.}
		\label{fig:genus2_train_track}
	\end{center}
\end{figure}

The induced action on cohomology, with the generators $\alpha_1, \alpha_2,
\beta_1,\beta_2$ and puncture curves $\delta_1,\delta_2$, is
\begin{equation*}
	\mathring\phi^*=\begin{bmatrix} 3 & -1 & -2 & 1 & -1 & 0\\ -1 & 3 & 1 & -2 & 1 & 0\\
		 -1 & 0 & 1 & 0 & 0 &0\\
		0 & -1 & 0 & 1 & -2 & 0\\ 0 & 0 & 0 & 0 & 1 & 0\\
		0 & 0 & 0 & 0 & 0 & 1\end{bmatrix}.
\end{equation*}
The matrix has largest eigenvalue
$\lambda_1 = \frac{5+\sqrt{21}}{2}$. The other eigenvalue $\lambda_2>1$ is
given by $\lambda_2 = \frac{3+\sqrt{5}}{2}$. The eigenvectors of $\phi^*$
for $\lambda_1$ and $\lambda_1^{-1}$ are
\begin{align*}
	\vec{e}_{\lambda_1} &= \begin{pmatrix}
		\frac{3+\sqrt{21}}{2} \\ -\frac{3+\sqrt{21}}{2} \\ -1 \\ 1\\ 0 \\0
		\end{pmatrix}\\
	\vec{e}_{\lambda_1^{-1}} &= \begin{pmatrix}
		-\frac{\sqrt{21}-3}{2} \\ \frac{\sqrt{21}-3}{2} \\ -1 \\ 1\\ 0 \\ 0
		\end{pmatrix}.
\end{align*}
We have a choice for $\vec{e}_{\lambda_1^{-1}}$ as it is only unique up to
scale.
We make the choice that is consistent with the orientation of the embedding
of Sol into $\mathbb{R}^4$. In particular, in the standard embedding, the
$x$-coordinate is contracted and the $y$-coordinate is expanded. Our choice
for $\vec{e}_{\lambda_1}$ and $\vec{e}_{\lambda_1^{-1}}$ has the same
orientation in the singular flat metric on $S$.

Thus, we obtain the parameters
\begin{align*}
	a_1 = -a_2 &= \frac{3+\sqrt{21}}{2}\\
	a_3 = -a_4 &=-1\\
	b_1 = -b_2 &= -\frac{\sqrt{21}-3}{2}\\
	b_3 = -b_4 &= -1.
\end{align*}

Fix a basepoint and choose representatives for $\alpha_1,\alpha_2,\beta_1,
\beta_2$ in $\pi_1(S)$, which we will also call $\alpha_1,\alpha_2,\beta_1,
\beta_2$ (see Figure \ref{fig:genus2_pi1}).
In addition, taking generators $\delta_1$ and $\delta_2$ for loops around the
singularities $s_1$ and $s_2$, we have the following action of $\phi$ on
$\pi_1(S\setminus \sigma)$:

\begin{figure}
	\begin{center}
		\includegraphics{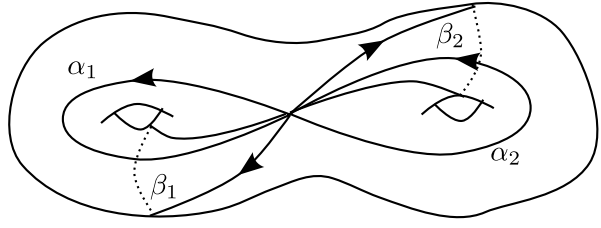}
		\caption{Generators for $\pi_1(S)$.}
		\label{fig:genus2_pi1}
	\end{center}
\end{figure}

\begin{align*}
	\phi(\alpha_1) &= \alpha_1 \beta_1^{-1}\delta_1^{-1}\alpha_2\beta_2
		\alpha_2^{-2}\alpha_1^2\beta_1^{-1}\\
	\phi(\alpha_2) &= \alpha_2^2\beta_2^{-1}\alpha_2^2\beta_2^{-1}
		\alpha_2^{-1}\delta_1\beta_1 \alpha_1^{-1} \\
	\phi(\beta_1) &= \beta_1\alpha_1^{-1}\\
	\phi(\beta_2) &= \alpha_1\beta_1^{-1}\delta_1^{-1}\alpha_2\beta_2
		\alpha_2^{-2}\beta_2\alpha_2\beta_2^{-1}\alpha_2^{-1}\delta_1^{-1}
		\beta_1\alpha_1^{-1}\\
	\phi(\delta_1) &= \delta_1\\
	\phi(\delta_2) &= \alpha_2 \beta_2 \alpha_2^{-2}\alpha_1 \beta_1^{-1}
		\delta_1^{-1} \delta_2 \delta_1 \beta_1 \alpha_1^{-1} \alpha_2^2
		\beta_2^{-1}\alpha_2^{-1}.
\end{align*}
with $a_5=a_6=b_5=b_6=0$.

Thus, we have that
\begin{align*}
	D &= \left[ \begin{smallmatrix}
		11+2 \sqrt{21} & \frac{-9-3\sqrt{21}}{2} & \frac{-17-3\sqrt{21}}{2} &
			7+2\sqrt{21} & \frac{-13-3\sqrt{21}}{2} & 0\\
		\frac{3+\sqrt{21}}{2} & 15-2\sqrt{21} & \frac{-3-\sqrt{21}}{2} &
			\frac{13+\sqrt{21}}{2} &\frac{-5-\sqrt{21}}{2} & 0\\
		\frac{3+\sqrt{21}}{2} & 0 & \frac{-3-\sqrt{21}}{2} & 0 & 0 & 0\\
		\frac{5+\sqrt{21}}{2} & 1 & \frac{-5-\sqrt{21}}{2} & -1 &
			\frac{5+\sqrt{21}}{2} & 0\\
		0 & 0 & 0 & 0 & 0 & 0\\
		0 & 0 & 0 & 0 & 0 & -5-\sqrt{21} \end{smallmatrix}\right]\\
	C & = \left[ \begin{smallmatrix}
		-62-13\sqrt{21} & \frac{125+5\sqrt{21}}{2} & \frac{101+21\sqrt{21}}{2} &
			-133-28\sqrt{21} & \frac{77+17\sqrt{21}}{2} & 0\\
		\frac{15+3\sqrt{21}}{2} & -103-20\sqrt{21} & \frac{-15-3\sqrt{21}}{2} &
			\frac{59+9\sqrt{21}}{2} & \frac{-23-5\sqrt{21}}{2} & 0\\
		\frac{15+3\sqrt{21}}{2} & 0 & \frac{-15-3\sqrt{21}}{2} & 0 & 0 & 0\\
		\frac{13+3\sqrt{21}}{2} & -4-\sqrt{21} & \frac{-13-3\sqrt{21}}{2} &
			19-4\sqrt{21} & \frac{23+5\sqrt{21}}{2} & 0\\
		0 & 0 & 0 & 0 & 0 & 0\\
		0 & 0 & 0 & 0 & 0 & 23+5\sqrt{21} \end{smallmatrix}\right]
\end{align*}
and $K=-2D$.
From this, we calculate from Equation \eqref{eqn:yparameters} that
\begin{equation*}
	\begin{pmatrix} y_1\\y_2\\y_3\\y_4\end{pmatrix} = -(\phi^*-I)^{-1}
		\left(D_{4\times 4}
	\begin{pmatrix}b_1\\b_2\\b_3\\b_4\end{pmatrix}+
	\begin{pmatrix} -y_5\\y_5 \\ 0 \\-2y_5\end{pmatrix}\right)
	= \begin{pmatrix} \frac{-3+\sqrt{21}}{2} \\ \frac{-3+\sqrt{21}}{2}-2y_5 \\
		-13+5\sqrt{21} - \frac{y_5}{3}\\ \frac{-53+17\sqrt{21}}{2}-\frac{5y_5}{3}\end{pmatrix},
\end{equation*}
and $y_5$ and $y_6$ are free. The span of $\phi^* - \lambda_1 I$ is generated
by the first three columns, so we can take $x_4=0$ (taking $x_4 \neq 0$ would
change the solution by a co-boundary). We then compute the other $x_i$ and
$y_0$ from Equation \eqref{eqn:xparameters}, yielding
\begin{align*}
	x_1 &= \frac{-18312+887\sqrt{21}}{42}
		+ \left(\frac{-3353+1121\sqrt{21}}{42}\right)y_5\\
	x_2 &= \frac{-2835+2573\sqrt{21}}{84}
		+\left(\frac{-812+40\sqrt{21}}{84}\right) y_5\\
	x_3 &= \frac{-2166+615\sqrt{21}}{6}+ \frac{-853+169\sqrt{21}}{6} y_5\\
	x_4 &= 0\\
	x_5 &= 0\\
	x_6 &= \left(\frac{6+2\sqrt{21}}{3}\right) y_6\\
	y_0 &= \frac{7119-1552\sqrt{21}}{84}+
			\left(\frac{1183-267\sqrt{21}}{84}\right)y_5.
\end{align*}

The $\omega$ entry in the commutator $\rho_{\HP}([\alpha_i,\beta_i])$ is computed
to be $2(a_i b_{2+i} - a_{2+i}b_{i})$. Hence, the total infintesimal cone
angle $\omega_{tot}$ is equal to $-4\sqrt{21}$. 
The infinitesimal cone angles about the
two boundary components should add up to $\omega_{tot} = -4\sqrt{21}$, and
the individual infinitesimal cone angles can be chosen so that the cone angles
about both singularities are decreasing towards $2\pi$.
By scaling the $b_i$ by a positive scalar, it is also possible
to change $\omega_{tot}$ to any negative number. 

\section{Discussion}

The hypotheses in Theorem \ref{thm:conestructures} are satisfied by 
pseudo-Anosov maps on the punctured torus, so the result includes
previously known case for the punctured torus. There exist examples
of pseudo-Anosov maps for other hyperbolic surfaces that satisfy the
conditions in the theorem. 

For an arbitrary pseudo-Anosov $\phi$, $\phi^*$ has 1 as an eigenvalue if and
only if the mapping torus $M_\phi$ has first Betti number $>1$. 
If $\phi^*$ does not have 1 as an eigenvalue but the
invariant foliations are not orientable, one can take an orientation cover for
the foliation and lift the pseudo-Anosov to the cover. However, this may
introduce additional eigenvalues for the lifted map. These conditions are needed to
prove Theorem \ref{thm:representations} in order to guarantee that an infinitesimal
deformation can be realized by a smooth path of deformed structures for small
time, but it would be interesting to know if the deformation can be carried out
even when the smoothness condition is not satisfied.

The result in Theorem \ref{thm:conestructures} is local -- we can
find a deformation of the cone angles for small time. It would be of further
interest to know whether the deformation can be carried out all the way to the
complete structure on $M_\phi$. This would give a direct connection between
the hyperbolic structure on fibered manifolds and the combinatorial properties
of the pseudo-Anosov monodromy.

\bibliographystyle{amsplain}
\bibliography{sources1}

\end{document}